\documentclass[11pt,reqno]{amsart}

\usepackage{a4wide}
\usepackage{amsmath,amssymb}
\usepackage{bbm}
\usepackage{xcolor}
\usepackage{graphicx}
\usepackage{ae,aecompl}
\usepackage{xcolor}
\usepackage{tcolorbox}

\newcommand\reallywidehat[1]{%
\savestack{\tmpbox}{\stretchto{%
  \scaleto{%
    \scalerel*[\widthof{\ensuremath{#1}}]{\kern-.6pt\bigwedge\kern-.6pt}%
    {\rule[-\textheight/2]{1ex}{\textheight}}
  }{\textheight}%
}{0.5ex}}%
\stackon[1pt]{#1}{\tmpbox}%
}

\newcommand{\ver}{{\vert\kern-0.25ex\vert\kern-0.25ex\vert }}

\allowdisplaybreaks

\theoremstyle{plain}
\newtheorem{theorem}{Theorem}[section]
\newtheorem{prop}[theorem]{Proposition}
\newtheorem{lemma}[theorem]{Lemma}
\newtheorem{coro}[theorem]{Corollary}

\theoremstyle{definition}
\newtheorem{definition}[theorem]{Definition}
\newtheorem{remark}[theorem]{Remark}

\newcommand{\Z}{{\mathbb Z}}

\newcommand{\R}{{\mathbb R}}
\newcommand{\N}{{\mathbb N}}
\newcommand{\C}{{\mathbb C}}

\newcommand{\XX}{{\mathbb X}}

\newcommand{\cA}{{\mathcal A}}
\newcommand{\cB}{{\mathcal B}}
\newcommand{\cC}{{\mathcal C}}
\newcommand{\cD}{{\mathcal D}}
\newcommand{\cF}{{\mathcal F}}
\newcommand{\cH}{{\mathcal H}}
\newcommand{\cU}{{\mathcal U}}
\newcommand{\cO}{{\mathcal O}}

\newcommand{\supp}{{\mbox{supp}}}

\newcommand{\card}{\mbox{card}}
\newcommand{\dd}{\mbox{d}}

\newcommand{\eps}{\varepsilon}
\newcommand{\cM}{{\mathcal M}}

\newcommand{\SAP}{\mathcal{S}\hspace*{-2pt}\mathcal{AP}}

\newcommand{\Bap}{\mathcal{B}\hspace*{-1pt}{\mathsf{ap}}}

\newcommand{\Cu}{C_{\mathsf{u}}}
\newcommand{\Cc}{C_{\mathsf{c}}}

\makeatletter
\newcommand{\ostar}{\mathbin{\mathpalette\make@circled\star}}
\newcommand{\make@circled}[2]{%
  \ooalign{$\m@th#1\smallbigcirc{#1}$\cr\hidewidth$\m@th#1#2$\hidewidth\cr}%
}
\newcommand{\smallbigcirc}[1]{%
  \vcenter{\hbox{\scalebox{0.77778}{$\m@th#1\bigcirc$}}}%
}
\makeatother

\begin{document}

\title[Almost Periodicity]{Abstract almost periodicity for group actions on uniform topological spaces.}

\author{Daniel Lenz}
\address{Mathematisches Institut, Friedrich Schiller Universit\"at Jena, 07743 Jena, Germany}
\email{daniel.lenz@uni-jena.de}
\urladdr{http://www.analysis-lenz.uni-jena.de}

\author{Timo Spindeler}
\address{Fakult\"at f\"ur Mathematik,
         Universit\"at Bielefeld, \newline
\hspace*{\parindent}Postfach 100131, 33501 Bielefeld, Germany}
\email{tspindel@math.uni-bielefeld.de}

\author{Nicolae Strungaru}
\address{Department of Mathematical Sciences, MacEwan University \\
10700 -- 104 Avenue, Edmonton, AB, T5J 4S2, Canada\\
and \\
Institute of Mathematics ``Simon Stoilow''\\
Bucharest, Romania}
\email{strungarun@macewan.ca}
\urladdr{http://academic.macewan.ca/strungarun/}

\begin{abstract}
We present  a unified theory for the almost periodicity of
functions with values in an arbitrary  Banach space, measures and
distributions via almost periodic elements for the  action of a
locally compact  abelian group on a uniform topological space. We
discuss the relation between Bohr and Bochner type almost
periodicity, and similar conditions, and how the equivalence among
such conditions relates to properties of the group action and the
uniformity. We complete the paper by demonstrating how various
examples considered earlier all fit in our framework.
\end{abstract}



\maketitle

\section{Introduction}

Almost periodicity plays an important role in many areas of
mathematics, from differential equations to aperiodic order. Let us
recall that $t$ is a period for a system if translating by $t$ the
entire system we obtain an identical copy of the original system. We
say that a system is fully periodic if the set of periods is
relatively dense.

As introduced by Bohr \cite{Boh}, the idea behind almost
periodicity is to replace the periods by almost periods. These are 
elements $t$ such that after translation the  system ``almost" agrees 
(with respect to a topology) with the original system. An element is
called almost periodic if the sets of almost periods are relatively
dense. Bohr's original definition was for the uniform convergence
topology on the space of uniformly continuous bounded functions
on the real line, and in this case the definition is
equivalent to the Bochner condition that the closed hull of the
function is a compact space in this topology. Many of these ideas
have been extended to other topologies on  spaces of functions on
various (or all) locally compact abelian groups by Stepanov
\cite{Ste}, Weyl \cite{Wey}, Besicovitch \cite{Bes}, to measures
\cite{bm,Gouere-1,ARMA,LSS} and even to distributions \cite{ST}.

Almost periodicity plays a fundamental role in the area of
Aperiodic Order due its connection to pure point diffraction. This
connection already appears (more or less explicit) in the work of
Meyer \cite{Mey}, Lagarias \cite{LAG} and Solomyak \cite{SOL,SOL1}.
A first study giving  an explicit equivalence statement  between
pure point spectrum and almost periodicity appears (in a specific
situation) in  \cite{bm} (see \cite{ARMA} as well for related work).
This was then extended and further studied in the framework of
dynamical systems in  \cite{Gouere-1,Gouere-2,MS,LS,NS11}. On a
fundamental level the equivalence was then established and studied
in an framework free of dynamical systems in  \cite{LSS}. It seems
fair to say that by now  the connection between pure point spectrum
and almost periodicity, as well as the importance of almost
periodicity to Aperiodic Order is well-established and thoroughly
understood.

\medskip

In many situations, the Bohr and Bochner definition of almost
periodicity are equivalent (see for example
\cite{Gouere-1,ARMA,MS,LS2} just to name a few). Moreover, the hull
of an almost periodic function/measure is often a compact Abelian
group \cite{LR,LS,LLRSS,MS}. Since the proofs in these related but
different situations are similar, it is natural to ask if there may
be any unified theory of almost periodicity which shows the
equivalence between Bohr and Bochner definition in a very general
situation. It is our goal here to answer this question.

Let us describe here our general approach. Both the Bohr and Bochner
definition for almost periodicity are in terms of properties of the
orbit of an element under the translation action of the group, and
can be defined for an arbitrary group action on a nice topological
space. While in most situations studied in literature the
topology is metrizable, this is neither the case in all situations
nor necessary. Here, we deal with the  more general case of group
actions on uniform spaces, that is where the topology is defined by
a uniformity. We should emphasize here that many of the results in
the paper can probably be extended to group actions on arbitrary
topological spaces. The uniformity structure seems to play an
important role when studying Bochner type almost periodicity, as in
this case pre-compactness is equivalent to total boundedness, and
this is the reason why the set up is of uniform spaces.

Given the action $\alpha$ of a locally compact Abelian group (LCAG) $G$ on a uniform space $X$, we study the connection between the following three properties of an element $x \in X$:
\begin{itemize}
  \item{} (Bohr type almost periodicity) For each entourage $U \in \cU$ the set $P_{U}(x):= \{ t \in G: (x, \alpha(t,x)) \in U \}$ of almost periods is relatively dense.
  \item{} (Bochner type almost periodicity) The orbit closure $\overline{ \{ \alpha(t,x): t \in G \}}$ is compact in $X$.
  \item{} (pseudo-Bochner type almost periodicity) For each entourage $U \in \cU$ the set $P_{U}(x)$ is finitely relatively dense.
\end{itemize}
Under the extra natural assumption that the uniformity is $G$-invariant (see Definition~\ref{def:G-inv} below), we show in Lemma~\ref{lem:AP-TB} that pseudo-Bochner type almost periodicity is simply the total boundedness of the orbit $\{ \alpha(t,x) : t \in G \}$. Therefore, for a $G$-invariant uniformity, Bochner type almost periodicity implies pseudo-Bochner type almost periodicity, and the two concepts are equivalent if the uniformity is also complete (Proposition~\ref{prop3}).
It is immediate from the definition that pseudo-Bochner type almost periodicity implies Bohr type almost periodicity. If the group action is equicontinuous (see Definition~\ref{def:equi}), we show in Proposition~\ref{prop-2} that the converse also holds and hence Bohr and pseudo-Bochner type almost periodicity are equivalent. Combining the results, and using the fact that for $G$-invariant actions, equicontinuity on orbit closure is equivalent to continuity at $0$ (see Lemma~\ref{lem:7}), we get in Theorem~\ref{ap-equiv} that for continuous, $G$-invariant actions on uniform spaces, Bohr, Bochner and pseudo-Bochner type almost periodicity are equivalent. Moreover, in this case, the orbit closure becomes a compact Abelian group (Theorem~\ref{thm:apequiv}).

We complete the paper by looking at some particular examples in Section~\ref{sect-exa}.

\section{Uniform topologies}
In this section, we discuss the necessary background on uniform
topologies needed for our considerations. The material is
well-known (see for example \cite[Ch.~2]{BOU} or \cite{Mel}). Let us start with the definition of uniformity.

\begin{definition}[Uniformity]
Let $X$ be a set. A nonempty set $\cU$ consisting of subsets of $X
\times X$ is called a \textbf{uniformity on $X$} if it satisfies the
following conditions.
\begin{itemize}
\item The diagonal
\[
\Delta := \{ (x,x)\, :\, x \in X \} \subseteq U
\]
of $X \times X$ is contained in any  $U \in \mathcal U$.
\item The set $\cU$ is closed upwards in the sense that if $U $ belongs to
$\cU$ and $U \subseteq V$ hold, then $V $ belongs to $\cU$ as well.
\item If $U,V$ belong to $\cU$, then $U \cap V$ belongs to $ \cU$.
\item For all $U \in \cU$ there exists a $V \in \cU$ such that
\[
V \circ V := \{ (x,y) \in X \times X\, :\, \mbox{ there exists }  z \in X
\mbox{ with  } (x,y), (y,z) \in V \}
\]
is contained in $U$.
\item For all $U \in \cU$ the set
\[
U^{-1}:= \{ (x,y)\, :\, (y,x) \in U \}
\]
belongs to $\cU$.

\end{itemize}
The elements $U \in \cU$ are called \textbf{entourages}.
\end{definition}

Any uniformity $\cU$ on $X$ induces a topology $\tau_{\cU}$ on $X$
as follows: For $x\in X$ and $U\in \cU$ we define
\[
U[x]:= \{ y \in X\, :\, (x,y) \in U \} \,.
\]
Then, the sets $\{ U[x] \,:\, U \in \cU \}$
define a basis at $x$ for the topology $\tau_{\cU} $. Specifically,
 a subset $O$ belongs to $\tau_{\cU}$ if and only if for any
$x\in O$ there exists  $U \in \cU$ with $U[x] \subseteq O$. We  say
that a topology is \textbf{uniform} if it is the topology induced by
a uniformity.

\smallskip

Any  metrisable topology is uniform. Indeed, if $(X,d)$ is a metric
space, for each $r>0$ we can define
\[
U_r:= \{ (x,y) \in X \times X\, :\, d(x,y) <r \} \,.
\]
Then,
\[
\cU_{d}:= \{ U \subseteq X \times X\, :\, \mbox{ there exists } r >0
\mbox{ such that } U_r \subseteq U \} \,,
\]
is a uniformity on $X$, which induces the topology defined by the
metric $d$. As it is instructive and may help the reader to get some
perspective on the definition of entourages we briefly sketch the
proof that $\cU_{d}$ is a uniformity: As $d(x,x) =0$ for all $x$, the
system $\cU_{d}$ contains the diagonal. The system $\cU_{d}$ is
upward closed by definition. The third point holds trivially. The
fourth point is a consequence of the triangle inequality and the last
point is a consequence of the symmetry of the metric.

\begin{remark}
In the preceding considerations we do not need $d$ to be a metric. It suffices that $d$ is a pseudometric, i.e. a
symmetric map $d : X\times X \longrightarrow [0,\infty]$ with
$d(x,x) =0$ for all $x\in X$ and
\[
d(x,y) \leq d(x,z) + d(z,y) \,,
\]
where the value $\infty$ is allowed.
\end{remark}

Let $\cU$ be a uniformity. Then, $\tau_{\cU}$ is Hausdorff if and
only if for all $x,y \in X$ with $x \neq y$ there exists some $U \in
\cU$ such that $(x,y) \notin U$.  Equivalently, $\tau_{\cU}$ is Hausdorff if and only if
\[
\bigcap_{U \in \cU} U =\Delta \,.
\]
All the uniform topologies we consider are assumed to be Hausdorff.

\medskip
We next discuss the completion of the topology induced by a uniformity. Let us start by recalling the definition of a filter.
\begin{definition} A \textbf{filter} on a set $X$ is a family $\cF \subseteq \mathcal{P}(X)$ of subsets of $X$ with the following properties:
\begin{itemize}
  \item{} If $A \in \cF$ and $A \subseteq B$ then $B \in \cF$.
  \item{} If $n \in \N$ and $A_1, \ldots, A_n \in \cF$ then $A_1 \cap \ldots \cap A_n \in \cF$.
  \item{} $\varnothing \notin \cF$.
\end{itemize}

If $\cU$ is a uniformity on $X$, a filter $\cF$ is called \textbf{Cauchy} if for all $U \in \cU$ there exists some $A \in \cF$ such that $A \times A \subseteq U$.

A uniform space $(X, \cU)$ is called \textbf{complete}  if every Cauchy filter is convergent.
\end{definition}

\medskip

Let us briefly discuss how to characterize uniform Hausdorff topologies which are given by a metric. First, we need the following definition.

\begin{definition}[Basis of an entourage]
Let $\cU$ be a uniformity on $X$. A set $\cB \subseteq \cU$ is
called a \textbf{basis of entourages} if for all $U \in \cU$ there
exists some $W \in \cB$ such that $W \subseteq U$.
\end{definition}

The following well-known result (see e.g. \cite{Weil}) characterizes the metrizability of a uniformity.

\begin{theorem}[Characterization metrizability of entourage]\label{thm-char-metrizability}
Let $\cU$ be a Hausdorff uniformity on $X$. Then, there exist a metric $d$ on $X$
such that $\cU=\cU_d$ if and only if there exists a countable basis
of entourages $\cB \subseteq \cU$. \qed
\end{theorem}

\smallskip

At the end of this section, let us review the notion of total boundedness and its connection to compactness.

\begin{definition}[Total boundedness]
Let $(X,\cU)$ be a uniform Hausdorff space. A set $A \subseteq X$
is called \textbf{totally bounded} if, for all $U \in \cU$, there
exist  $U_1,U_2,\ldots,U_n \subseteq X$ such that
\[
  A  \subseteq \bigcup_{j=1}^n U_j \qquad \mbox{ and } \qquad \bigcup_{j=1}^n (U_j \times U_j)   \subseteq U  \,.
\]
\end{definition}

The importance of totally boundedness is given by the following result.

\begin{theorem}\label{thm-tot-bound}\cite[Thm.~II.4.3]{BOU}
Let $(X,\cU)$ be a uniform Hausdorff space.
\begin{itemize}
\item[(a)] If $A \subseteq X$ is compact in $\tau_{\cU}$, then $A$ is totally bounded.
\item[(b)] If $A \subseteq X$ is totally bounded and $(X, \cU)$ is complete, then the closure
  $\overline{A}$  of $A$ is compact.
\end{itemize} \qed
\end{theorem}

\begin{remark}[Characterization compactness] From the theorem we may easily derive the following characterization of compactness in a uniform space.
A subset $A\subseteq X$ is compact if and only if $A$ is totally bounded and
complete (in the topology inherited from $\cU$).
\end{remark}

The topology of a compact set is always uniform.

\begin{theorem}\cite[Thm.~II.4.1]{BOU}
Let $(X, \tau)$ be a compact space. Then, there exists a unique
uniformity $\cU$ on $X$ such that $\tau=\tau_{\cU}$.  \qed
\end{theorem}

\section{Group actions on uniform topological spaces}
In this section, we consider groups acting on uniform spaces. The
material is certainly known.

Let $(X, \cU)$ be a uniform Hausdorff space, and let $\tau_{\cU}$ be
the topology induced by $\cU$ on $X$. Let $G$ be a group. Then, a
map  $\alpha : G \times X \to X$ is called a \textbf{group action}
if it satisfies:
\begin{itemize}
\item $\alpha(e,x) = x$ for all $x\in X$ (where $e$ is the neutral
element of $G$).
\item $\alpha(s,\alpha(t,x)) = \alpha(st,x)$ for all $x\in X$ and
$t,s\in G$.
\end{itemize}
If $\alpha : G\times X\longrightarrow X$ is a group action, the
group $G$ is said to act on $X$  (via $\alpha$).

Next, we review and introduce various definitions for the action
$\alpha$, which will play a central role in this paper. Note that we
do not at first assume any continuity property of $\alpha$. Indeed,
we have not assumed that $G$ carries a topology so far.

\medskip

We start by introducing the notion of $G$-invariance for
uniformities and (pseudo) $G$-invariance for metrics on $X$.

\begin{definition}\label{def:G-inv}
Let $\alpha$ be an action of $G$ on $X$.
\begin{itemize}
\item[(a)]
A uniformity $\cU$ on $X$ is called \textbf{$G$-invariant} if for
each  $U \in \cU$ there exists a $U' \in \cU$ such that, for all $t \in G$, we have
\[
\alpha(t,U'):= \{ (\alpha(t,x), \alpha(t,y)) :  (x,y) \in U' \}
\subseteq U \,
\]

  \item[(b)] A metric $d$ on $X$ is called \textbf{$G$-invariant} if, for all $x,y \in X$ and $t \in G$, we have
\[
d(\alpha(t,x), \alpha(t,y))=d(x,y) \,.
\]
\item[(c)] A metric $d$ on $X$ is called \textbf{pseudo $G$-invariant} if for each $\eps >0$ there exists a $\delta>0$ such that, for all $x,y \in X$ with $d(x,y)< \delta$ and all $t \in G$, we have
\[
d(\alpha(t,x), \alpha(t,y))< \eps \,.
\]
\end{itemize}
\end{definition}

It is obvious that any $G$-invariant metric is pseudo $G$-invariant.
It is also clear from  the definition in (a) that $U'$ itself is
contained in $U$ (as we can take $t$ to be the neutral element of
$G$). This will be used tacitly subsequently.

\smallskip

Note that another way to phrase (c) of the preceding definition is
that the family $\alpha_t$, $t\in G$, (with $\alpha_t (x) =
\alpha(t,x)$) is equicontinuous. By changing the metric to an
equivalent metric, we can then make $\alpha_t$ into a families of
isometries. This is discussed next.

\begin{lemma}[Invariance and pseudo invariance]\label{metris}
Let $\alpha$ be an action of $G$ on a metric space $(X,d)$, and let $\cU_{d}$ be the uniformity on $X$ defined by $d$.
Then, the following statements are equivalent.
\begin{itemize}
\item[(i)] The uniformity $\cU_{d}$ is $G$-invariant.
\item[(ii)] The metric $d$ is pseudo $G$-invariant.
\item[(iii)] There exists a $G$-invariant metric $d'$ on $X$
with $\cU_{d} = \cU_{d'}$.
\end{itemize}
\end{lemma}
\begin{proof}
(i)$\implies$(ii): Let $\eps >0$. Since $U_\eps := \{
(x,y) \,:\, d(x,y) < \eps \} \in \cU$ and $\cU$ is $G$ invariant, there
exists some $V \in \cU$ such that, for all $t \in G$ and $(x,y) \in
V$, we have $(\alpha(t,x), \alpha(t,y)) \in U_\eps$.

Now, since $V \in \cU$, there exists some $\delta>0$ such that
$d(x,y) < \delta$ implies $(x,y) \in V$.
Therefore, we have $(x,y) \in V$ for all $x,y \in X$ with $d(x,y) <\delta$ and all $t \in G$, and hence $(\alpha(t,x), \alpha(t,y)) \in
U_\eps$. This implies
\[
d(\alpha(t,x), \alpha(t,y)) <\eps \,.
\]

\medskip

\noindent (ii)$\implies$(iii): Define
\[
\bar{d}(x,y):=\sup\{ d(\alpha(t,x),\alpha(t,y)) \,:\, t \in G \} \,.
\]
It follows immediately from the definition that we have
\begin{equation}\label{eq1}
\bar{d}(x,y)=\bar{d}(\alpha(t,x),\alpha(t,y))
\end{equation}
for all $x,y \in X$ and $t \in G$. Now, $\bar{d}$ may not be a metric. Indeed, since $X$ is not necessarily compact, $d$ could be unbounded, i.e there could exist some $x,y \in X$ such that $\bar{d}(x,y)=\infty$. To fix this issue, define
\[
d'(x,y) :=
\begin{cases}
\bar{d}(x,y), & \mbox{ if } \bar{d}(x,y) <1 \,,\\
  1, &\mbox{ otherwise }.
\end{cases}
\]
We claim that this is a metric which has the desired properties.

\begin{itemize}
\item{} We show that $d'$ is a metric.

First note that $d'(x,y) \geq 0$ for all $x,y \in G$ by construction.
Moreover, $d'(x,y)=0 \iff \bar{d}(x,y)=0 \iff
d(\alpha(t,x), \alpha(t,y))=0 \mbox{ for all } t \in G
\iff x=y$.

Next, the definition of $\bar{d}$ gives that $\bar{d}(x,y)=\bar{d}(y,x)$ for all $x,y$ and hence $d'(x,y)=d'(y,x)$.

Finally, let us prove the triangle inequality. Let $x,y, z \in X$. We split the problem into three cases:

{\it Case 1:} Assume $d'(x,z) \geq 1$. Then, one finds
\[
d'(x,y) \leq 1 \leq d'(x,z) \leq d'(x,z)+d'(y,z) \,.
\]

{\it Case 2:} Assume $d'(y,z) \geq 1$. Then, we get
\[
d'(x,y) \leq 1 \leq d'(y,z) \leq d'(x,z)+d'(y,z) \,.
\]

{\it Case 3:} Assume $d'(x,z) < 1$ and $d'(y,z) <1$. Then, $d'(x,z)=\bar{d}(x,z)$ and  $d'(y,z)=\bar{d}(y,z)$ imply
\begin{align*}
d'(x,y)
    &\leq \bar{d}(x,y)= \sup\{ d(\alpha(t,x),\alpha(t,y))\, :\, t \in G \}   \\
    &\leq \sup\{ d(\alpha(t,x),\alpha(t,z)) +d(\alpha(t,y),\alpha(t,z))\,:\,
       t \in G \} \\
    &\leq  \sup\{ d(\alpha(t,x),\alpha(t,z))\, :\, t \in G \}
      +\sup\{d(\alpha(t,y),\alpha(t,z))\, :\, t \in G \}\\
    &= \bar{d}(x,z)+\bar{d}(y,z) = d'(x,z)+d'(y,z) \,.
\end{align*}
This proves that $d'$ is a metric.

\smallskip

\item{} $d'$ is $G$-invariant.

This follows immediately from \eqref{eq1} and the definition of $d'$.

\item{} $d$ and $d'$ induce the same uniformity.

First note that we have
\[
d'(x,y) <r \implies
\bar{d}(x,y) <r \implies d(x,y) <r
\]
for all $0<r<1$. Hence,
\[
\{(x,y) \in X\times X \,:\, d'(x,y)<r\}\subseteq \{(x,y)\in X\times X\, :\,
d(x,y) <r\}
\]
for all $0< r <$. This gives, $\cU_{d} \subseteq \cU_{d'}$.

\smallskip

Next, let $\varepsilon >0$ be given. Assume without loss of
generality that $\varepsilon <1$. Since $d$ is pseudo-invariant,
there exists some $\delta
>0$ such that $d(\alpha(t,y),\alpha (t,z)) <\varepsilon$ for all $y,z \in  X$ with $d(y,z) <\delta$ and all $t
\in G$. In
particular, we have
\[
d'(x,y) \leq \varepsilon
\]
whenever $d(x,y)<\delta$. This gives
\[
\{(x,y) \in X\times X \,:\, d(x,y) <\delta\} \subseteq \{(x,y)\in
X\times X\, :\, d'(x,y) \leq \varepsilon\} \,.
\]
This gives $\cU_{d'}\subseteq \cU_{d}$, and hence $\cU_{d} = \cU_{d'}$.
\end{itemize}

\medskip

\noindent (iii)$\implies$(i): Since $\cU_{d}$ agrees with $\cU_{d'}$ and
$d'$ is $G$-invariant, it follows immediately that $\cU_{d}$ is $G$-invariant.
\end{proof}

Next, consider a uniform Hausdorff space $(X,\cU)$, and let $\tau_{\cU}$
be the topology defined by $\cU$. Let $G$ be a locally compact
Abelian group (LCAG). We then write the group operation as addition
with $+$, and denote the neutral element of $G$ as $0$ and the
inverse of $t\in G$ by $-t$.

Let $\alpha$ be an action of  $G$ on $X$. Recall that $\alpha$ is
called continuous if
\[
\alpha : G \times X \to X
\]
is a continuous mapping with respect to the topologies of $G$ and $(X ,\tau_{\cU})$.

We want to look at a stronger version of continuity for $\alpha$. To do so, we
start with the following lemma.

\begin{lemma}\label{lem6}
Let $(X,\cU)$ be a Hausdorff uniform space
and $\alpha : G\times X\longrightarrow X$ the continuous action of
an LCAG.  The following assertions are equivalent.
\begin{itemize}
\item[(i)] For all $U \in \cU$ there exists an open set $O\subseteq G$ containing
$0$ and a set $V \in \cU$ such that
\[
\{ (\alpha(s,x), \alpha (t,y))\, :\, s,t \in O ,\, (x,y) \in V \} \subseteq U \,.
\]
\item[(ii)] For all $U \in \cU$ there exists an open set $O\subseteq G$ containing $0$ and a set $V \in \cU$ such that
\[
\{ (\alpha(s,x), y)\, :\, s \in O ,\, (x,y) \in V \} \subseteq U \,.
\]
  \item[(iii)] For all $U \in \cU$ there exists an open set $O\subseteq G$ containing $0$ such that
\[
\{ (\alpha(s,x),x) \,:\, s \in O ,\, x \in X \} \subseteq U \,.
\]
\item[(iv)] The family of function $\{ \alpha_x \}_{x \in X }$ defined by $\alpha_x : G \to X$
\[
\alpha_x (t)= \alpha(t,x)
\]
is equicontinuous at $t=0$ (i.e. for every $U\in \cU$ there exists
an open set $O\subseteq G$ containing $0$ with $(\alpha_x (s),x) =
(\alpha_x(s),\alpha_x(0))\in U$ for all $s\in O$).
\item[(v)] The family of function $\{ \alpha_x \}_{x \in X }$ is uniformly  equicontinuous (i.e. for every $U\in \cU$ there exists
an open set $O\subseteq G$ containing $0$ with $(\alpha_x(s),\alpha_x(t))\in
U$ for all $s,t\in G$ with $s-t\in O$).
\end{itemize}
\end{lemma}
\begin{proof}
(i)$\implies$(ii): This is obvious, since
\[
\{ (\alpha(s,x), y) \, :\, s \in O , (x,y) \in V \} \subseteq \{ (\alpha(s,x), \alpha (t,y))\, :\, s,t \in O ,\, (x,y) \in V \} \,.
\]

\smallskip

\noindent (ii)$\implies$(iii): This is obvious, since $\Delta = \{ (x,x) \, : \, x \in X \} \subseteq V$.

\medskip

\noindent (iii)$\implies$(i): Let $U \in \cU$ and let $W'$ be so that $W' \circ W' \circ W' \subseteq U$. Let
\[
W= W' \cap (W')^{-1} \,.
\]
Then, by (iii), there exists an open set $O\subseteq G$ containing $0$ such that
\[
\{ (\alpha(s,x), x) \, :\, s \in O ,\, x \in X \} \subseteq W \,.
\]
Now,  we have $(\alpha(s,x),
x) \in W$ for all $s,t \in O$ and $(x,y) \in W'$ and $(\alpha(t,y), y) \in W$. Since $W = W' \cap (W')^{-1}$, we obtain
\[
  (\alpha(s,x), x) \in W' \,, \qquad
  (x,y) \in W'   \qquad   \mbox{ and } \qquad
  (y,\alpha(t, y))  \in W' \,.
\]
This gives  $(\alpha(s,x), \alpha (t,y)) \in W' \circ W' \circ W' \subseteq
U$. This shows that
\[
 \{ (\alpha(s,x), \alpha (t,y)) \,:\, s,t \in O ,\, (x,y) \in W' \} \subseteq U \,.
\]

\smallskip

\noindent (iii)$\implies$(v): Let $U \in \cU$.
By (iii), there exists an open set $O\subseteq G$ containing $0$ such that
\begin{equation}\label{eq3}
\{ (\alpha(s,y), y) \, :\, s \in O ,\, y \in X \} \subseteq U \,.
\end{equation}
Now, let $s-t \in O$. Then, for all $x \in X$, setting
$y=\alpha(t,x)$ in \eqref{eq3} we get
\[
  (\alpha(s-t,\alpha(t,x)),\alpha(t,x)) \in U \,, \qquad
   (\alpha(s,x)),\alpha(t,x))   \in U \qquad \mbox{ and } \qquad
  (\alpha_x(s), \alpha_x(t)) \in U \,.
\]
This shows that $(\alpha_x(s), \alpha_x(t))\in U $ holds for all $x \in X$ and all $s-t \in O$, proving the equicontinuity of this family.

\medskip

\noindent (v)$\implies$(iv): This is obvious.

\medskip

\noindent (iv)$\implies$(iii): This is obvious.
\end{proof}

We can now define the notion of equicontinuous group actions.

\begin{definition}[Equicontinuous group action]\label{def:equi}
Let $\alpha$ be an action of $G$ on the uniform Hausdorff space $(X,\cU)$. We say that $\alpha$ is \textbf{equicontinuous} if it satisfies one (and thus all) the equivalent conditions of Lemma~\ref{lem6}.
\end{definition}

\begin{remark}
We will discuss examples in Section~\ref{sect-exa} where  the action is not continuous, or continuous but not equicontinuous.
\end{remark}

\begin{remark}
The action $\alpha$ is continuous if, for all $U \in \cU$ and $x \in X$, the set
\[
O_{x,U}=\{ t \in G \,:\, (\alpha(t,x),x) \in \cU \}
\]
is an open neighbourhood of $0$ in $G$.

The action $\alpha$ is equicontinuous if, for all $U \in \cU$, the set
\[
O_U=\{ t \in G\, :\, (\alpha(t,x),x) \in U  \mbox{ for all } x \in X
\}=\bigcap_{x \in X} O_{x,U}
\]
is an open neighbourhood of $0$ in $G$.
\end{remark}

\medskip

\begin{definition}[Orbit and hull]
Let $\alpha$ be an action of $G$ on the uniform Hausdorff space
$(X,\cU)$. For $x\in X$ the \textbf{orbit of $x$} $O(x)$ is defined as
\[
O(x) := \{ \alpha(t,x) \,:\, t \in G \} \,.
\]
The \textbf{hull of $x$} (also called the
\textbf{orbit closure of $x$}), denoted by $H_x^{\cU}$, is defined as the closure of the orbit in $(X, \tau_{\cU})$, i.e.
\[
H_x^{\cU} := \overline { \{ \alpha(t,x) \,:\ t \in G \}}^{\tau_{\cU}} \,.
\]
\end{definition}

Next, let us show that for $G$-invariant uniformities, the continuity of the action implies the equicontinuity on orbit closures.

\begin{lemma}[Equicontinuity on orbit closures] \label{lem:7}
Let $\alpha$ be an action of $G$ on the
uniform Hausdorff space $(X,\cU)$. Assume that $\alpha$ is continuous and
$\cU$ is $G$-invariant. Then, for each $x \in X$, $\alpha$ is equicontinuous on
$H_x^{\cU}$.
\end{lemma}
\begin{proof}
Let $U \in \cU$, and let $V \in \cU$ be such that $V \circ V \circ V
\circ V\subseteq U$. Since $\alpha$ is $G$-invariant, there exists some $V' \in \cU$ such that
\begin{equation}\label{eq21}
\{ (\alpha(t,x), \alpha(t,y)) \,:\, t \in G,\, (x,y) \in V' \} \subseteq V \,.
\end{equation}
Clearly, $V'\subseteq V$ holds (as we can take $t=0$ on the left hand
side).
Since $\alpha$ is continuous, there exists an open set $O
\subseteq G$ containing $0$ such that
\begin{equation}\label{eq22}
(\alpha(t,x), x) \in V'
\end{equation}
for all $ t \in O$.
Now, let $t \in O$ and $y,z \in H_x$ be arbitrary such that $(y,z) \in V'$.
Since $y,z \in H_z$ there exist $s,r \in G$ such that $(y, \alpha(s,x)) \in V'$ and $(z, \alpha(r,x)) \in V'$.
Then, since $(y, \alpha(s,x)) \in V'$, we have by \eqref{eq21}
\[
(\alpha(t,y),\alpha(s+t,x)) \in V \,.
\]
Moreover, since $t \in O$, by \eqref{eq22}
\[
(\alpha(t,x),x) \in V'
\]
and hence, again by
by \eqref{eq21}
\[
(\alpha(t+s,x),\alpha(s,x)) \in V \,.
\]
Therefore,
\[
(\alpha(t,y), \alpha(s,x)) \in V \circ V  \,.
\]
Finally, since $(y, \alpha(s,x)) \in V' \subseteq V$, we get
\[
(\alpha(t,y), y) \in V \circ V \circ V \,.
\]
Using $(y,z) \in V' \subseteq V$, we obtain
\[
(\alpha(t,y), z) \in V \circ V \circ V \circ V \subseteq U \,.
\]
Therefore, we have
\[
(\alpha(t,y), z)  \in U
\]
for all $t \in O$ and $y,z \in H_x$ with $(y,z) \in V'$. By Lemma~\ref{lem6}, the action $\alpha$ is equicontinuous on $H_x$.
\end{proof}

\section{Bohr and Bochner type almost periodicity}
For this entire section we let a LCAG $G$ be given and $\alpha$ is
an action of $G$ on a uniform Hausdorff space $(X, \cU)$. We discuss
the standard definitions of almost periodicity in this context.

\smallskip

Let $A \subseteq G$ be any set. Then, $A$ is \textbf{relatively
dense} if there exists a compact set $K$ such that
\[
A+K =G \,.
\]
We say that $A$ is \textbf{finitely relatively dense} if there exists a
set finite $F \subseteq G$ such that $A+F=G$. Clearly, a finitely
relatively dense set is relatively dense.

\smallskip

We can now introduce the following definitions.

\begin{definition}\label{def:ap-type}
For $x \in X$ and $U \in \cU$ define
\[
P_{U}(x):=\{ t \in G \,:\, (\alpha(t,x), x) \in U \} \,.
\]
\begin{itemize}
\item[(a)] The element $x \in X$ s called \textbf{Bohr-type almost periodic}
if for all $U \in \cU$ the set $P_U(x)$ is relatively dense.
\item[(b)] The element $x \in X$ is called \textbf{Bochner-type almost
periodic} if $H_x^\cU$ is compact in $(X, \tau_{\cU})$.
\item[(c)] The element $x \in X$ is called \textbf{pseudo Bochner-type almost
periodic}\footnote{ We will see below that this definition is closely related to the total boundedness of the orbit $O_x$, and hence to Bochner type almost periodicity.} if for all $U \in \cU$ the set $P_U (x)$ is finitely
relatively dense.
\end{itemize}
\end{definition}

\smallskip

\begin{remark}
\begin{itemize}
\item[(a)]
Let $\alpha$ be an action of $G$ on a uniform Hausdorff space
$(X,\cU)$ such that the uniformity  is $G$-invariant.
 Then, for all $U \in \cU$ there exists some $V \in \cU$ such that,
for all $x \in X$ and $t \in G$, we have
\[
  P_{V}(x) \subseteq P_{U}(\alpha(t,x))  \qquad \text{ and } \qquad
  P_V(\alpha(t,x)) \subseteq  P_U(x) \,.
\]
\item[(b)] Consider the space $\Bap_{2,\cA}(G)$ of Besicovitch almost periodic functions on a LCAG $G$ (see \cite{LSS} for details) together with the uniformity $\cU$ given by the metric $d(f,g):= \| f-g \|_{b,2,\cA}$, and the action $\alpha(t,f)=\tau_tf$ (see \cite{LSS}). Then, an element $f \in \Bap_{2,\cA}(G)$ is Bohr-type almost periodic if and only if $f$ is a mean almost periodic function.

We will see below more examples, where Bohr/Bochner-type almost periodicity is equivalent with some other standard notion of almost periodicity. For this reason, to avoid confusion, we used the names Bohr-type and Bochner-type almost periodicity in Definition~\ref{def:ap-type}, instead of the simpler Bohr or Bochner almost periodicity.
\end{itemize}
\end{remark}


\smallskip

Next, we will see that there is a connection between the total boundedness of the orbit of an element $x$ and pseudo Bochner type almost periodicity.

\begin{lemma}\label{lem:AP-TB}
Let $\alpha$ be an action of $G$ on $(X, \cU)$ and let $\cU$ be $G$-invariant. Then, an element $x\in X$ is pseudo Bochner-type almost periodic if and only if the orbit $O(x)$ is totally bounded.
\end{lemma}
\begin{proof}
We start with a preliminary observation used
repeatedly in the proof. Let $U \in \cU$. Due to the $G$-invariance, there
exists a $V\in \cU$ with $(\alpha(s,x),\alpha(s,y))\in U$ for all
$(x,y)\in V$ and $s\in G$. Hence, $(\alpha(t,x),\alpha(r,x))\in V$
for some $r,s\in G$ implies
\[
(\alpha(t-r,x),x)=(\alpha(-r,\alpha(t,x)),\alpha(-r,\alpha(r,x))) \in U \,,
\]
which in turn
gives
\[
t\in r + P_U (x) \,.
\]

\medskip

\noindent $\Longleftarrow$: Let $U\in \cU$ be given. Choose $V\in \cU$ with
$(\alpha(s,x),\alpha(s,y))\in U$ for all $(x,y)\in V$ and $s\in G$.
Since $O(X)$ is totally bounded, there exists some $V_1,\ldots, V_n
\subseteq X$ such that
\[
O(X) \subseteq V_1\cup \ldots \cup V_n \qquad \mbox{ and } \qquad
  \bigcup_{j=1}^n (V_j \times V_j) \subseteq V  \,.
\]
Without loss of generality we can assume  $O(X) \cap V_j \neq
\varnothing$, as otherwise we can erase $U_j$ from our list.
Therefore, for each $1 \leq j \leq n$ there exists some $t_j \in G$
such that $\alpha(t_j,x) \in V_j$. We claim
\[
G =  \bigcup_{j=1}^n (t_j + P_U(x))
\]
(which gives finitely relative denseness of $P_U (x)$ and, hence,
Bochner-type-almost periodicity of $x$).
Indeed, let $t \in G$. Then,
\[
\alpha(t,x) \in O(X) \subseteq V_1\cup \ldots \cup V_n
\]
holds and,  hence, there exists some $1 \leq j \leq n$ such that $\alpha(t,x) \in V_j$. Therefore,
\[
(\alpha(t,x), \alpha(t_j,x)) \in V_j \times V_j \subseteq V
\]
holds and
\[
t \in  P_U(\alpha(t_j,x)) \,.
\]
follows from the observation at the beginning of the proof.

\medskip

\noindent $\Longrightarrow$: Let $U \in \cU$ be given.
Let $W \in \cU$ be given with $W \subseteq U$ and $W \circ W^{-1}
\subseteq U$. Choose $V\in \cU$ with $(\alpha(s,x),\alpha(s,y))\in W$
for all $s\in G$ and $(x,y)\in V$. Since $x$ is pseudo Bochner-type
almost periodic, there exist $t_1, \ldots ,t_n$ with
\[
G = \bigcup_{j=1}^n (t_j + P_V(x)) \,.
\]
For each $1 \leq j \leq n$ define
\[
U_j= \{ y \in X \,:\, (y, \alpha(t_j,x)) \in W \}  \,.
\]

Now, let $t \in G$ be arbitrary. Then, there exists some $1 \leq j
\leq n$ with $t \in t_j +  P_V(x)$ and hence
\[
(\alpha(t-t_j,x),x) \in V \,.
\]
This gives
\[
(\alpha(t,x), \alpha(t_j,x))\in W
\]
and,  hence,
 $\alpha(t,x) \in U_j$. As $t\in G$ was arbitrary, we infer
\[
O(X) \subseteq \bigcup_{j=1}^n U_j \,.
\]
It remains to show $U_j \times U_j \subseteq U$ for $j =1,\ldots, n$.
So, let $1 \leq j \leq n$ and $(y,z) \in U_j \times U_j$ be given.
Then $(y, \alpha(t_j,x)),  (z, \alpha(t_j,x)) \in W$ giving $(y,
\alpha(t_j,x)) \in W$ and $  (\alpha(t_j,x), z) \in W^{-1}$ and,
hence,
\[
(y,z) \in W \circ W^{-1} \subseteq U \,.
\]
This shows that $U_j \times U_j \subseteq U$ for all $j$ and the
proof is finished.
\end{proof}

By combining this result with Theorem~\ref{thm-tot-bound}, we obtain the following connection between Bochner-type almost periodicity and pseudo Bochner-type almost periodicity.

\begin{prop}\label{prop3}
Let $G$ act on a uniform topological Hausdorff space and let $\cU$
be $G$-invariant. Then, the following holds for $x\in X$:
\begin{itemize}
  \item[(a)] If $x$ is Bochner-type almost periodic then $x$ is pseudo Bochner-type almost periodic.
  \item[(b)] If $(X, \tau_U)$ is complete and $x$ is pseudo Bochner-type almost periodic then $x$ is Bochner-type almost periodic.
\end{itemize}
\qed
\end{prop}

In fact, it is easy to see that the equivalence between Bochner-type and pseudo Bochner-type almost periodicity is a matter of the completeness of $\cU$.

\begin{coro}
Let $G$ act on a uniform topological Hausdorff space
and assume that the uniformity is $G$-invariant. Let  $x \in X$ be
pseudo-Bochner almost periodic. Then $x$ is Bochner almost periodic
if and only if $H_x^{\cU}$ is complete.
\end{coro}

\medskip

By definition, any pseudo Bochner-type almost periodic point is
Bohr-type almost periodic (as a finitely relatively dense set is clearly
relatively dense). Next, we want to show that for equicontinuous
group actions, Bohr-type almost periodicity implies pseudo Bochner-type almost periodicity. This will allow us show, in the case of equicontinuous group actions of $G$ on a complete $G$-invariant uniform Hausdorff space
$(X, \cU)$, the equivalence between Bohr-type, Bochner-type and
pseudo Bochner-type almost periodicity.

\begin{prop}\label{prop-2}
Let $G$ act on a uniform topological Hausdorff space $(X,\cU)$ such
that the action is equicontinuous, and let $x \in X$. Then, $x$ is
Bohr-type almost periodic  (i.e. $P_U(x)$ is relatively dense for
any $U\in\cU$) if and only if $x$ is pseudo-Bochner-type almost
periodic (i.e. for all $U \in \cU$ the set $P_U(x)$ is finitely
relatively dense).
\end{prop}
\begin{proof}
$\Longleftarrow$: This  follows immediately from the fact that every
finite set in $G$ is compact.

\medskip

\noindent $\Longrightarrow$: Let $U \in \cU$ be arbitrary. Since the action is
equicontinuous, there exists an open set $O\subseteq G$ and $V \in
\cU$ such that
\begin{equation}\label{eq2}
\{ (\alpha(s,x), \alpha (t,y))\,:\, s,t \in O ,\, (x,y) \in V \} \subseteq U \,.
\end{equation}

Since $x$ is Bohr-type almost periodic, there exists a compact set $K$ such that
\[
P_V(x)+K = G \,.
\]
Next, since $K$ is compact, and $O$ is open, there exists some finite set $F$ such that
\[
K \subseteq F+O \,.
\]
We will show that
\[
G=P_U(x)+F \,.
\]

Indeed, let $t \in G$. Then, there exists some $s \in P_V(x)$ and $k \in K$ such that $t=s+k$.  Since $k \in K \subseteq F+O$, there exists some $f \in F$ and $u \in O$ such that
\[
k=f+u \,.
\]
Therefore
\[
t=s+f+u \,.
\]

We claim that $s+u\in P_U(x)$, which will complete the proof. Indeed, $s
\in P_V(x)$ implies that
\[
(\alpha(s,x), x) \in V
\]
Since $u,0  \in O$ and $(\alpha(s,x), x) \in V$ by \eqref{eq2} we have
\[
(\alpha(u,\alpha(s,x)), \alpha(0,x))=(\alpha(u+s,x), x) \in U
\]
and hence $s+u \in P_U(x)$ as claimed.
\end{proof}

The next result is a consequence of the previous proposition.

\begin{theorem}[Main result - I] \label{ap-equiv}
Let $G$ act on a complete uniform topological Hausdorff
space $(X,\cU)$ such that the action is continuous and $\cU$ is $G$-invariant.
Let $x \in X$. Then, the following statements are equivalent.
\begin{itemize}
  \item[(i)] $x$ is Bohr-type almost periodic.
  \item[(ii)] $x$ is Bochner-type almost periodic.
  \item[(iii)] $x$ is pseudo Bochner-type almost periodic.
\end{itemize}
\end{theorem}
\begin{proof}
Note that $\cU$ induces a uniformity $\cU_x$ on $H_x^{\cU}$, which is complete and $G$-invariant. By Lemma~\ref{lem:7}, the action $\alpha$ is equicontinuous on $(H_x^\cU, \cU_x)$.
Moreover, it is easy to see that $x$ is Bohr-type, Bochner-type or
pseudo Bochner-type almost periodic in $(X, \cU)$, respectively, if
and only if $x$ is Bohr-type, Bochner-type or pseudo Bochner-type
almost periodic in $(H_x^{\cU}, \cU_x)$.

The equivalence of (i) and (iii) now  follows from
Proposition~\ref{prop-2} and the equivalence of (ii) and (iii) follows
from  Proposition~\ref{prop3}  (applied to $(H_x^\cU, \cU_x)$).
\end{proof}

We complete the section by discussing how -  if  the uniformity is
$G$-invariant and complete -  $H_x^{\cU}$ has a natural Abelian
group structure.

\begin{prop}
Let $G$ act on a uniform topological Hausdorff space $(X,\cU)$ such
that the uniformity $\cU$ is $G$-invariant and $\alpha$ is
continuous. Let $x \in X$ be given such that  $H_x^{\cU}$ is
complete. Then, the following statements hold.
\begin{itemize}
  \item[(a)]
 $H_x^{\cU}$ is a topological  Abelian group with the addition $\oplus$ induced by
\[
\alpha(t,x) \oplus \alpha(s,x):= \alpha(s+t,x) \,.
\]
  \item[(b)] $F: G \to H_x^{\cU}\,$ via
\[
F(t)= \alpha(t,x)
\]
is a uniformly continuous group homomorphism, with dense range.
\end{itemize}
\end{prop}
\begin{proof}
Note first that $\alpha$ is equicontinuous on
$H_x^{\cU}$ by Lemma~\ref{lem:7}.

\medskip

\noindent (a) We first show that $\oplus$ defined by
\[
\alpha(t,x) \oplus \alpha(s,x):= \alpha(s+t,x)
\]
is well defined on $O(x)$, defines an Abelian group structure, and
addition and inversion are continuous with respect to the topology
induced by $\cU$.

\smallskip

To show that addition is well defined consider $t,t',s,s'\in G$ with
$\alpha(t,x)=\alpha(t',x)$ and $\alpha(s,x)=\alpha(s',x)$. Then a
direct computation gives
\[
\alpha(t+s,x)=\alpha(t, \alpha(s,x))= \alpha(t, \alpha(s',x)) =\alpha(s', \alpha(t,x))=\alpha(s', \alpha(t',x))=\alpha(t'+s',x) \,.
\]
This shows that $\oplus$ is well defined. By definition, $O(x)$ is
closed under $\oplus$.

\smallskip

Associativity of the addition on $G$ yields that $\oplus$ is
associative. Indeed, a direct computation gives
\[
\left( \alpha(t,x) \oplus \alpha(s,x) \right) \oplus \alpha(r,x) =\alpha( (t+s)+r,x)=\alpha( t+(s+r),x)  \\
=\alpha(t,x) \oplus \left( \alpha(s,x)  \oplus \alpha(r,x)\right)
\,.
\]
Since $G$ is Abelian, it is obvious that $\oplus$ is commutative.
Moreover, for all $t \in G$, we have
\[
\alpha(t,x) \oplus \alpha(0,x)= \alpha(t,x) \,.
\]
This shows that $x=\alpha(0,x)$ is the identity in $O(x)$. Finally, for all $t \in G$, we have
\[
\alpha(t,x) \oplus \alpha(-t,x)= \alpha(0,x) \,.
\]
This shows that $(O(x), \oplus )$ is an Abelian group.

\medskip

We now show that $(O(x), \oplus)$ becomes a topological group when
equipped with the topology induced by $\cU$.

Indeed, let $U \in \cU$ be any entourage. Let $W$ be such that $W \circ W \subseteq U$. Since $\alpha$ is $G$-invariant, there exists some $V \in \cU$ such that, for all $(y,z) \in V$ and $t \in G$, we have
\begin{equation}\label{eq32}
(\alpha(t,y), \alpha(t,z)) \in W \,.
\end{equation}
Now, for all $\alpha(t_1,x), \alpha(t_2,x), \alpha(s_1,x), \alpha(s_2,x) \in O_x$ with
\[
  (\alpha(s_1,x), \alpha(s_2,x))  \in V \qquad \text{ and } \qquad
  (\alpha(t_1,x),\alpha(t_2,x))  \in V \,,
\]
\eqref{eq32} implies
\[
  (\alpha(s_1+t_1,x), \alpha(s_2+t_1,x))  \in W \qquad \text{ and } \qquad
  (\alpha(t_1+s_2,x),\alpha(t_2+s_2,x))  \in W
\]
and hence
\[
(\alpha(s_1+t_1,x), \alpha(s_2+t_2,x)) \in U  \,.
\]
It follows that
\[
\big( \alpha(s_1,x) \oplus \alpha(t_1,x) ,  \alpha(s_2,x) \oplus \alpha(t_2,x) \big) \in U \,.
\]
This proves that $ \oplus : O(x) \times O(x) \to O(x) $ is uniformly
continuous with respect to $\cU$.

We now turn to inversion. As already discussed, any $\alpha(t,x)
\in O(x)$ has the inverse
\[
\ominus \alpha(t,x)= \alpha(-t,x) \,.
\]
Let $U \in \cU$ be arbitrary, and let $V=U^{-1}$.  Since $\alpha$ is $G$-invariant, there exists some $W \in \cU$ such that, for all $(y,z) \in V$ and $t \in G$, we have
\begin{equation}\label{eq33}
(\alpha(t,y), \alpha(t,z)) \in V \,.
\end{equation}
Now, if $(\alpha(t,x), \alpha(s,x)) \in V$, then we have by \eqref{eq33}
\[
(\alpha(-s-t, \alpha(t,x)),\alpha(-s-t, \alpha(s,x))) \in V=U^{-1}
\]
and hence
\[
(\alpha(-t,x),\alpha(-s,x)) \in U \,.
\]
This show $\ominus : G \to G$ is uniformly continuous with respect to $\cU$.

\smallskip

Now, $O(x)$ is a dense subset of the complete set $H_x^{\cU}$.
Therefore, by \cite[Ch. II, Prop.13]{BOU}, $H_x^{\cU}$
is -what is known as - the Hausdorff completion of $O(x)$. Since
$(O(x), \oplus, \cU)$ is a topological Abelian group, we then
infer by \cite[Ch. III, Thm. I, Thm.~II]{BOU} that $H_x^{\cU}$ is an Abelian group and the inclusion $i:O(x) \to H_x^{\cU}$ is a group homomorphism.

\medskip

\noindent (b) It is obvious that $F' : G\longrightarrow O(x)$, $t\mapsto
\alpha(t,x)$ is a group homomorphism as is the inclusion
\[
i : O(x)\longrightarrow H_x^{\cU} \,, \qquad y\mapsto y \,.
\]
Thus, we have the following group homomorphisms:
\[
G \stackrel{F'}{\longrightarrow} O(x) \stackrel{i}{\hookrightarrow}
H_x^{\cU} \,.
\]
Since the inclusion $i$ is uniformly continuous, to complete the
proof we need to show that $F'$ is uniformly continuous.

Let $U \in \cU$ be arbitrary. Since $\cU$ is $G$-invariant, there exists some $U' \in \cU$ such that
\begin{equation}\label{eq11}
\{ (\alpha(t,x), \alpha(t,y)) \,:\, t \in G,\, (x,y) \in U' \} \subseteq U \,.
\end{equation}

Next, since $\alpha$ is continuous, it is equicontinuous on
$H_x^{\cU}$ by Lemma \ref{lem:7}. Thus, there exists an open set $O\subseteq G$ containing $0$ and $V \in \cU$ such that
\begin{equation}\label{eq233}
\{ (\alpha(s,x), \alpha (t,y)) \,:\, s,t \in O ,\, (x,y) \in V \cap
H_x^{\cU} \times H_x^{\cU} \} \subseteq U' \,.
\end{equation}
Now, let $s,t \in G$ be so that $s-t \in U$. Then , by \eqref{eq233}
we have $(\alpha(s-t,x), \alpha (0,x)) \in U'$ and hence, by
\eqref{eq11} we have
\[
(F'(t), F'(s))=  (\alpha(t,\alpha(s-t,x)), \alpha(t,\alpha(0,x)))
\in U \,.
\]
This shows that $F'$ is uniformly continuous, and completes the
proof.
\end{proof}

\begin{remark}
(a) Let us emphasize that we do not need compactness (i.e. almost
periodicity of any form)  of $H_x^{\cU}$ to obtain the group
structure on $H_x^{\cU}$.  To illustrate this we include the
following example. Consider the translation action of $\R$ on the
vector space $\Cu(\R)$ of uniformly continuous and bounded functions
equipped with the supremum norm $ \| \cdot \|_\infty$. For any
nontrivial $f\in \Cu(\R)$ with $\lim_{x\to \pm\infty} f(x) =0$ is
easy to see that $H_f$ is isomorphic to $\R$, under the canonical
isomorphism $t\in G\leftrightarrow T_tf \in H_f$.

\medskip

\noindent (b) Since $F' : G\longrightarrow O(x)$  is onto, by the fundamental
theorem of group isomorphism, we have
\[
O(x) \simeq G/\mbox{Per(x)} \,.
\]
\end{remark}

In the case that the hull is compact we can even say more.

\begin{coro}\label{thm-bochenr-conseq}
Let $G$ act on a uniform topological Hausdorff space $X$ such that
the uniformity $\cU$ is $G$-invariant and $\alpha$ is continuous. If
$x \in X$ is Bochner-type almost periodic, then  $H^{\cU}_x$ is a
compact Abelian group with the addition $\oplus$ induced by
\[
\alpha(t,x) \oplus \alpha(s,x):= \alpha(s+t,x) \,
\]
and $F : G\longrightarrow H_x^{\cU}$, $t\mapsto \alpha(t,x)$, is a
continuous group homomorphism with dense range.
\end{coro}
\begin{proof}
Since $x$ is Bochner-type almost periodic, $H_x^{\cU}$ is compact.
Hence, it is also  complete. Now, the statement follows from
Proposition~ref{p1}.
\end{proof}

The previous corollary can be understood
in terms of the \textbf{Bohr-compactification} $G_{\mathsf{b}}$ of
$G$. This is the (unique) compact group with the universal property
that there exists a continuous group homomorphism $\iota_b :
G\longrightarrow G_{\mathsf{b}}$ such that, for any continuous group
homomorphism $\psi : G\longrightarrow H$ into a compact group $H$,
there exists a unique mapping $\Psi : G_{\mathsf{b}}\longrightarrow H$ with
$\psi = \Psi \circ \iota$ (see for example \cite[Prop.~4.2.6]{MoSt}
for details).

\begin{coro}\label{coro-bochner-bohr-compactification}
Let $G$ act on a uniform topological Hausdorff space $X$ such that
the uniformity $\cU$ is $G$-invariant and $\alpha$ is continuous. If
$H^{\cU}_x$ is a compact Abelian group with the addition $\oplus$
induced by
\[
\alpha(t,x) \oplus \alpha(s,x):= \alpha(s+t,x) \,,
\]
then there
exists a surjective continuous group homomorphism $\Psi :
G_{\mathsf{b}}\longrightarrow H_x^{\cU}$ such that
\[
\alpha(t,x) = \Psi(\iota_{\mathsf{b}}(t))
\]
holds for all $t\in G$.
\end{coro}
\begin{proof}
Since $H_x^{\cU}$ is a compact Abelian group, by the universal
property there exists a continuous group homomorphism $\Psi :
G_{\mathsf{b}} \to H^{\cU}_x$ such that
\[
\alpha(t,x)=F(t)= \Psi( i_{\mathsf{b}}(t)) \quad \mbox{ for all } t \in
G \,.
\]
Since $F$ has dense range, so does $\Psi$. Furthermore, since
$G_{\mathsf{b}}$ is compact, so is $F(G_{\mathsf{b}})$. It follows
that the range of $F$ is compact and hence closed. In particular,
$F$ is surjective.
\end{proof}

Combination of the previous results gives our next main result.

\begin{theorem}[Main result - II]
\label{thm:apequiv} Let $G$ act
on a complete uniform Hausdorff topological $X$ space such that the
action is continuous and the uniformity $\cU$ is $G$-invariant.
Then, for $x \in X$, the following assertions are equivalent.
\begin{itemize}
  \item[(i)] $x$ is Bohr-type almost periodic.
  \item[(ii)] $x$ is Bochner-type almost periodic.
  \item[(iii)] $x$ is pseudo Bochner-type almost periodic.
  \item[(iv)] $H_x^{\cU}$ is a compact Abelian group with the group operation $\oplus$ induced by
\[
\alpha(t,x) \oplus \alpha(s,x):= \alpha(s+t,x) \,.
\]
\item[(v)] There exists a continuous function $\Psi : G_{\mathsf{b}} \to H^{\cU}_x$ such that
\[
\alpha(t,x)= \Psi( i_{\mathsf{b}}(t)) \quad \mbox{ for all } t \in G
\,.
\]
\end{itemize}

Moreover, in this case $F(t)=\alpha(t,x)$ define a group homomorphism with dense range $F : G \to H^{\cU}_x$, and $\Psi$ is an onto group homomorphism.
\end{theorem}
\begin{proof}
The equivalences between (i), (ii) and (iii) follow from
Theorem~\ref{ap-equiv}. The implication (ii)$\implies$(iv)
follows from Corollary~\ref{thm-bochenr-conseq}. This Corollary
gives also that $F : G\longrightarrow H_x^{\cU}$, $t\mapsto
\alpha(t,x)$, is a continuous group homomorphism with dense range.
The implication (iv)$\implies$(v) follows from Corollary
\ref{coro-bochner-bohr-compactification}. Next, we will show
(v)$\Longrightarrow$(ii):
 Since $\Psi$ is continuous, and $G_{\mathsf{b}}$ is compact, so is
$\Psi(G_{\mathsf{b}})=:K$. Moreover, by  (v) we also have
$\alpha(t,x)= \Psi( i_{\mathsf{b}}(t))$ for all $t\in G$, which
implies $O(x)\subset K$. As the compact $K$ is closed we infer
\[
H_x^{\cU} = \overline{O(x)}\subseteq K \,.
\]
Thus, $H_x^{\cU}$ must be compact as it is a closed subset of a
compact set. This shows (ii). Moreover, as $K$ is contained in
$H_x^{\cU}$ (by definition of $\Psi$) we have even

\[
H_x^{\cU} \subseteq  K = \Psi(G_{\mathsf{b}}) \subseteq H_x^{\cU}
\]
giving $H_x^{\cU} =K$ and hence $\Psi$ is onto. This completes the proof.
\end{proof}

\section{The mixed uniformity}\label{mix}
In the preceding discussion and most notably in Lemma \ref{lem:7} we
have seen that the continuity of $\alpha$ is equivalent to the
equicontinuity of $\alpha$ on  orbit closures (provided  $\cU$ is
$G$-invariant). This - so to say - booster of continuity  has been a
main tool in our considerations. Now, it may happen that $\alpha$ is
not continuous. However,  even if we lack continuity, by a simple
mixing process we can make $\alpha$ equicontinuous, without changing
the Bohr-type almost periodicity, or the other basic properties of
the uniformity. This is discussed in this section.

Note, however, that since the equivalence between Bohr-type and
pseudo Bochner-type almost periodicity relies on equicontinuity,
this mixing process can actually change pseudo-Bochner and Bochner
type almost periodicity, and we will see such an example in the next
section.

\medskip

Let $G$ be a LCAG and let $\cU$ be a $G$-invariant uniformity on
some set $X$ and $\alpha : G \times X \to X$ any group action.
Denote by $\cO$ the set of all open sets in $G$ containing $0$.

Now, for $U \in \cU$ and $O \in \cO$ define
\[
V[O,U]:= \{ (\alpha(t,x), y) \,:\, t \in O,\, (x,y) \in U \} \,,
\]
and set
\[
\cU_{\mathsf{mix}} := \{ V \subseteq X \times X : \mbox{ there
exists } U \in \cU, O \in \cO \mbox{ such that } V[O,U] \subseteq V \}
\,.
\]
Let us first show that $\cU_{\mathsf{mix}}$ is indeed a uniformity on $X$.

\begin{prop}[$\cU_{\mathsf{mix}}$ as uniformity]
If $\cU$ is $G$-invariant, then $\cU_{\mathsf{mix}}$ is a uniformity
on $X$.
\end{prop}
\begin{proof} We have to show that the five points defining a
uniformity are satisfied. It is rather straightforward to show this.
For the convenience of the reader we include a proof.

\begin{itemize}
  \item{} Let $V \in \cU_{\mathsf{mix}}$. Then, there exists some $O \in \cO$ and $U \in \cU$ so that $V[O,U] \subseteq V$. Then, for all $x \in X$ we have $0 \in O$ and $(x,x) \in U$ and hence
\[
(x,x) = (\alpha(0,x),x) \in V[O,U] \subseteq V \,.
\]
This shows that $\Delta \subseteq V$.
\item{} Let $V \in \cU_{\mathsf{mix}}$ and $V \subseteq W$. Since $V \in \cU_{\mathsf{mix}}$, there exists some $O \in \cO$ and $U \in \cU$ such that $V[O,U] \subseteq V \subseteq W$. This shows that $W \in \cU_{\mathsf{mix}}$.
\item{} Let $V_1,V_2 \in \cU$. Then, there exists some $O_1, O_2 \in \cO$ and $U_1, U_2 \in \cU$ such that $V[O_j,U_j] \subseteq V_j$.
 We have $O:= O_1 \cap O_2 \in \cO, U=U_1 \cap U_2 \in \cU$. Moreover, it follows immediately from the definition that for $1 \leq j \leq 2$ we have
\[
V[O,U] \subseteq V[O_j,U_j] \,.
\]
Therefore,
\[
V[O,U] \subseteq V[O_1,U_1] \cap V[O_2,U_2] \subseteq V_1 \cap V_2 \,,
\]
which shows that $V_1 \cap V_2 \in \cU_{\mathsf{mix}}$.
\item{} Let $V \in \cU_{\mathsf{mix}}$. Then, there exists some $O \in \cO$ and $U \in \cU$ so that $V[O,U] \subseteq V$.
Then, since $O \in \cO$ we can find some $O' \in \cO$ so that $O'+O' \subseteq O$. Since $U \in \cU$, there exists some $U_1 \in \cU$ such that $U_1 \circ U_1 \subseteq U$.

By the $G$-invariance of $\cU$, there exists some $U' \in \cU$ such
that
\begin{equation}\label{eq4}
\{ (\alpha(t,x), \alpha(t,y)) : t \in G , (x,y) \in U' \} \subseteq
U_1 \,.
\end{equation}
Note that $U'$ is contained in $U$ (as we can set $t =0$)  and set
$V':= V[O',U']$.

Let $(x,y) \in   V'\circ V'$. Then, there exists some $z \in X$ such that $(x,z), (z,y) \in V'=V[O',U']$. Therefore, there exists some $x',z' \in X$ and $t,s \in O$ such that
\[
(x',z), (z',y) \in U' \,, \qquad
  x =\alpha(t,x') \qquad \mbox{ and } \qquad
  z= \alpha(s,z') \,.
\]
Note here that $z'=\alpha(-s,\alpha(s,z'))=\alpha(-s,z)$ and
$x'=\alpha(-t,x)$ hold.
Therefore, by \eqref{eq4} we have $(\alpha(-t-s,x), \alpha(-s,z)) \in U' \subseteq U_1$ and $(\alpha(-s,z),y) \in U'\subseteq U_1$ and hence $(\alpha(-t-s,x),y) \in U_1 \circ U_1 \subseteq U$.
This gives
\[
(x,y) = (\alpha(t+s,\alpha(-t-s,x)),y) \in V[O,U] \subseteq V \,.
\]
Hence, we obtain $V' \circ V' \subseteq V$.
\item{} Let $V \in \cU_{\mathsf{mix}}$. Then, there exists some $O \in \cO$ and $U \in \cU$ so that $V[O,U] \subseteq V$.
Next, since $U \in \cU$ and the $\cU$ is $G$-invariant, there exists some $U' \in \cU$ such that
\[
\{ (\alpha(t,x), \alpha(t,y))\, :\, t \in G ,\, (x,y) \in U' \} \subseteq U \,.
\]

Now, let $(x,y) \in V[-O, (U')^{-1}]$. Then, there exists some $t \in -O$ and $x' \in X$ such that $(x',y) \in (U')^{-1}$ and $x= \alpha(t,x')$.
Next, $(x',y) \in (U')^{-1}$ gives $(y,x') \in U'$ and hence $(\alpha(t,y), x) \in U$. Therefore, since $-t \in O$ we have
\[
(y,x)=(\alpha(-t, \alpha(t,y)), x) \in V[O,U] \subseteq V
\]
and hence $(x,y) \in V^{-1}$.
This proves that
\[
V[-O, (U')^{-1}] \subseteq V^{-1}
\]
and hence $V^{-1} \in \cU_{\mathsf{mix}}$.
\end{itemize}
\end{proof}

\begin{definition}[Mixed uniformity]
Let the LCAG $G$ act on the
uniform space  $(X,\cU)$ such that $\cU$ is a $G$-invariant. Then,
$\cU_{\mathsf{mix}}$ is called the \textbf{mixed uniformity}
induced from the action of $G$.
\end{definition}

\begin{prop}[Characterization of $\cU_{\mathsf{mix}}$]
Let the LCAG $G$ act on the uniform space  $(X,\cU)$ such that $\cU$
is $G$-invariant. Then, the following statements hold.
\begin{itemize}
\item[(a)] The action $\alpha$ is equicontinuous on $(X, \cU_{\mathsf{mix}})$ and
   $\cU$ is finer than $\cU_{\mathsf{mix}}$, that is $\cU_{\mathsf{mix}} \subseteq
   \cU$ holds.
\item[(b)] $\cU$ is the finest uniformity satisfying (a), i.e. whenever
  $\cU'$ is a uniformity such that $\cU$ is finer than $\cU'$
and
  $\alpha$ is equicontinuous on $(X,\cU')$, then
$\cU'\subseteq \cU_{\mathsf{mix}}$ holds.
\end{itemize}
In particular $\alpha$ is equicontinuous on $(X,\cU)$ if and only if
$\cU = \cU_{\mathsf{mix}}$.

\end{prop}
\begin{proof} The last statement is immediate from (a) and  (b).

\medskip

\noindent (a) From Lemma~\ref{lem6} (iii) it follows easily that
$\alpha$ is equicontinuous with respect to $(X,\cU_{\mathsf{mix}})$.
To show that $\cU$ is finer that $\cU_{\mathsf{mix}}$
we consider an arbitrary  $V \in \cU_{\mathsf{mix}}$. Then, there
exists some $U \in \cU$ and $O \in \cO$ such that
\[
V[O,U] \subseteq V \,.
\]
Now, clearly $U \subseteq V[O,U]$ holds and we get $U \subseteq V$.
Since $\cU$ is a uniformity, we get $V \in \cU$ and therefore
$\cU_{\mathsf{mix}} \subseteq \cU \,. $ follows as $V\in
\cU_{\mathsf{mix}}$ was arbitrary.

\medskip

\noindent (b) Let $\cU'$ with $\cU'\subseteq \cU$ be given such that $\alpha$ is
equicontinuous on $(X,\cU')$.  Let $V \in \cU'$ be arbitrary. Since
$\alpha$ is equicontinuous on $(X,\cU')$, by Lemma~\ref{lem6} (ii),
there exists some $O \in \cO$ and $U \in \cU$ such that
\[
V[O,U] \subseteq V \,.
\]
As $\cU_{\mathsf{mix}}$ is a uniformity it is upward closed and $V
\in \cU_{\mathsf{mix}}$ follows. As $V\in \cU'$ was arbitrary, the
desired inclusion $\cU'\subseteq \cU_{\mathsf{mix}}$ follows.
\end{proof}

We now turn to the question of metrizability of
$\cU_{\mathsf{mix}}$.

\begin{prop}[Metrizability of $\cU_{\mathsf{mix}}$] If $\cU$ and the topology on $G$ are metrisable, then $\cU_{\mathsf{mix}}$ is metrisable.
\end{prop}
\begin{proof}
Since $\cU$ is metrisable, we can find a countable basis of
entourages $\cB \subseteq \cU$. Moreover, since the topology on $G$
is metrisable, we can find some countable $\cC \subseteq \cO$ which
is a basis of open sets at $t=0$ for the topology of $G$.
It follows immediately that $\{ V[O,U] \,:\, O \in \cC,\, U \in \cB \}$ is
a countable bases of entourages for $\cU_{\mathsf{mix}}$. The
desired statement now follows from Theorem~\ref{thm-char-metrizability}.
\end{proof}

\begin{remark}
The converse is not true. Indeed, consider the action of any group $G$ on an arbitrary set $X$. Let $\cU$ be the uniformity on
$X$ given by the discrete metric $d$.

It is easy to see that
\[
\cU= \{ U \subseteq X \times X : \Delta \subseteq X \} \,.
\]
Then, it follows immediately that $\cU_{\mathsf{mix}}=\cU$ is
metrisable, no matter if the topology on $G$ is metrisable or not.
\end{remark}

\smallskip

Now, we show that if $\cU$ is $G$-invariant, complete and
metrisable, $\cU_{\mathsf{mix}}$ is complete.

\begin{prop} Let $\cU$ be a $G$-invariant uniformity. If $\,\cU$ and the topology on $G$ are metrisable and if $\cU$ is complete, then $\cU_{\mathsf{mix}}$ is complete.
\end{prop}
\begin{proof}
Since $\cU$ is $G$-invariant and metrisable, by Lemma~\ref{metris}
there exists a $G$-invariant metric $d_{\cU}$  which defines the
uniformity $\cU$. Next, let $d_G$ be any $G$-invariant metric on $G$
which gives the topology of $G$.
Since $\cU_{\mathsf{mix}}$ is metrisable (by the preceding
proposition), to prove completeness it suffices to show that every
sequence $(x_n) \in X$ which is Cauchy with respect to
$\cU_{\mathsf{mix}}$ has a subsequence which is convergent in
$\cU_{\mathsf{mix}}$ to some $x \in X$.

Let $(x_n)$ be a Cauchy sequence with respect to
$\cU_{\mathsf{mix}}$. Let
\begin{align*}
  O_n & = \{ t \in G \,:\, d_G(t,0) < 2^{-n}\} =B_{2^{-n}}(0) \,, \\
  U_n & =\{ (x,y) \in X \times X \,:\, d_{\cU}(x,y) <2^{-n} \} \in \cU  \,, \\
  V_n & = V[O_n,U_n] \in \cU_{\mathsf{mix}} \,.
\end{align*}
We note
\[
V_n = \{ (\alpha(t,x), y) \,:\, d_U(x,y) < 2^{-n} ,\, d_G(t,0)
< 2^{-n}\} \,.
\]
Now, since $(x_n)$ is Cauchy with respect to $\cU_{\mathsf{mix}}$,
for each $m \in \N$ there exists some $M(m)\in \N$ such that, for
all $n_1,n_2 >M(m)$ we have
\[
(x_{n_1}, x_{n_2}) \in V_m \,.
\]
Hence, we can construct inductively a sequence $(n_k)$ such that
\[
  n_1 > M(1) \qquad \text{ and } \qquad
  n_{k+1} > \max \{ n_k , M(1),M(2),\ldots, M(k),M(k+1) \}  \,.
\]
Note here that $n_{k+1} > n_k$ by construction and hence $(x_{n_k})$ is a subsequence of $(x_n)$.
Next, since $n_{k}, n_{k+1} >M_k$ the definition of $M_k$ gives
\[
(x_{n_{k+1}}, x_{n_k}) \in V_k \,.
\]
Therefore, there exist some $t_k \in O_k$ and $y_k$ such that
\[
  (x_{n_{k+1}}, x_{n_k})   = (\alpha(r_k,y_k), x_{n_k}) \qquad \text{ and } \qquad
  d_U(y_k,x_k)  <2^{-k} \,.
\]
Let $t_k := -r_k$. Then $d(t_k,0)=d(0,r_k) <2^{-k}$, and $x_{n_{k+1}}=\alpha(r_k,y_k)$ gives
\[
y_k= \alpha(-r_k,x_{n_{k+1}})=  \alpha(t_n,x_{n_{k+1}}) \,.
\]
In particular,
\[
 d_G(t_k,0)  <2^{-k} \qquad \text{ and } \qquad
d_U(\alpha(t_k,x_{n_{k+1}}), x_{n_k} ) =   d_U(y_k,x_k)  <2^{-k} \,.
\]

Define $s_{k}=t_1+t_2+\ldots +t_{k-1} \in G$. Then, since $d_G$ is $G$-invariant, we have
\[
d(s_{k+1},s_k)=d(s_{k+1}-s_k,0)=d(t_k,0) < 2^{-k}
\]
and a standard telescopic argument shows that $(s_k)$ is a Cauchy sequence in $G$, and hence, convergent to some $s \in G$.

Next, since the metric $d_U$ is $G$-invariant, we have
\[
d_U(\alpha(s_{k+1},x_{n_{k+1}}), \alpha(s_{k},x_{n_{k}}))=d_U(\alpha(s_{k+1}-s_k,x_{n_{k+1}}), x_{n_{k}})=d_U(\alpha(t_k,x_{n_{k+1}}), x_{n_k}  < 2^{-k} \,.
\]
Again, by a standard telescopic argument, the sequence $(y_k)$ with $y_k=\alpha(s_{k}, x_{n_k})$ is a Cauchy sequence in $(X, d_U)$. Since $\cU$ is complete it follows that $(y_k)$ converges to some $y \in X$.

Now, since $(x_{n_k})$ with $x_{n_k}=\alpha(-s_k,y_k), y_k$ converges in $\cU$ to $y\in X$ and $(s_k)$ converges in $G$ to $s$, the definition of
$\cU_{\mathsf{mix}}$ implies immediately that $(x_{n_k})$ converges
to $\alpha(-s,y)$. This completes the proof.
\end{proof}

Next, we show that mixing the uniformity does not change Bohr-type almost periodicity for an element.

\begin{theorem}
Let $\alpha$ be an action of $G$ on a uniform space $(X, \cU)$ such
that $\cU$ is $G$-invariant.  Then, $x\in X$ is Bohr-type almost
periodic with respect to $\cU$ if and only if $x$ is Bohr-type
almost periodic with respect to $\cU_{\mathsf{mix}}$.
\end{theorem}
\begin{proof}
$\Longrightarrow$: This follows immediately from the definition of
Bohr-type almost periodicity and $\cU_{\mathsf{mix}} \subseteq \cU$.

\medskip

\noindent $\Longleftarrow$: Let an arbitrary $U \in \cU$ be given. Fix $O \in
\cO$ with compact closure and  consider
\[
V=V[O,U] \in
\cU_{\mathsf{mix}} \,.
\]
  Since  $x$ is Bohr-type almost periodic with
respect to $\cU_{\mathsf{mix}}$, the set
\[
P_{V}(x):=\{ t \in G \,:\, (\alpha(t,x), x) \in V \}
\]
is relatively dense.  Hence, there exists some compact set $K
\subseteq G$ such that
\[
P_{V}(x)+K =G \,.
\]
Now, let $t \in P_{V}(x)$
be arbitrary. This means
\[
(\alpha(t,x), x) \in V[O,U]
\]
and therefore, there exists some $s \in O$ and $y \in X$ such that
\[
  \alpha(t,x)  =\alpha(s,y) \mbox{ and }
  (y,x) \in U
  \]
holds. The first relation gives $y=\alpha(t-s,x)$ and therefore,
\[
(\alpha(t-s,x),x) \in U \,.
\]
This shows
\[
t-s \in P_U(x)=\{ t \in G \,:\, (\alpha(t,x), x) \in U \} \,.
\]
Thus we obtain
\[
t \in P_U(x)+s \subseteq P_U(x)+O \subseteq P_U(x)+\overline{O} \,.
\]
Since $t \in P_{V}(x)$ was arbitrary, we get
\[
P_{V}(x) \subseteq P_U(x)+\overline{O} \,,
\]
and hence
\[
G= P_{V}(x)+K \subseteq P_U(x)+(\overline{O}+K) \,.
\]
Since $\overline{O}+K$ is compact, we get that $P_U(x)$ is
relatively dense. As $U \in \cU$ was arbitrary, the claim follows.
\end{proof}

When we now collect all results from this section and combine them with Theorem~\ref{thm:apequiv}, the following result is an immediate consequence.

\begin{coro}\label{cor1}
Let $G$ act on a complete metric space $(X,d)$ space such that $d$
is pseudo $G$-invariant, and let $\cU=\cU_d$. Then, for $x \in X$,
the following assertions are equivalent.
\begin{itemize}
  \item[(i)] $x$ is Bohr-type almost periodic in $(X, \cU, \alpha)$.
  \item[(ii)] $x$ is Bohr-type almost periodic in $(X, \cU_{\mathsf{mix}}, \alpha)$.
  \item[(iii)] $x$ is Bochner-type almost periodic in $(X, \cU_{\mathsf{mix}}, \alpha)$.
  \item[(iv)] $x$ is pseudo Bochner-type almost periodic in $(X, \cU_{\mathsf{mix}}, \alpha)$.
  \item[(v)] $H_x^{\cU_{\mathsf{mix}}}$ is a compact Abelian group with the group operation $\oplus$ induced by
\[
\alpha(t,x) \oplus \alpha(s,x):= \alpha(s+t,x) \,.
\]
\item[(vi)] There exists a continuous function $\Psi : G_{\mathsf{b}} \to H^{\cU_m}_x$ such that
\[
\alpha(t,x)= \Psi( i_{\mathsf{b}}(t)) \quad \mbox{for all }  t \in G
\,.
\]
\end{itemize}
Moreover, in this case $F(t)=\alpha(t,x)$ define a group
homomorphism with dense range $F : G \to H^{\cU}_x$, and $\Psi$ is
an onto group homomorphism.   \qed
\end{coro}

\section{Examples}\label{sect-exa}
In this section we present a wealth of examples for our results.
This will in particular how all earlier corresponding results on
almost periodicity that we are aware of fall within our framework.

\subsection{Bohr and Bochner almost periodicity for functions}
Consider $X$ to be the vector space $\Cu(G)$ of uniformly continuous
bounded functions on a LCAG $G$ with the topology given by $\| \cdot
\|_\infty$ and the translation action
\[
\alpha(t,f)(x):= f(x-t) \,.
\]
It is easy to see that $\alpha$ is $G$-invariant, equicontinuous
and that the uniformity is complete. Therefore, Thm.~\ref{thm:apequiv} gives the following well-known result.
\begin{theorem}\label{CU-equiv} Let $f \in \Cu(G)$. Then, the following are equivalent.
\begin{itemize}
  \item[(i)] $f$ is Bohr-almost periodic.
  \item[(ii)] $f$ is Bochner-almost periodic.
  \item[(iii)] $H_f:=\overline{ \{ T_tf \,:\,t \in G \}}$ is a compact Abelian group with the addition operation induced by
\[
(T_tf) \oplus (T_sf)= T_{s+t} f\,.
\]
\item[(iv)] There exists a continuous function $\Psi : G_{\mathsf{b}} \to H_f$ such that
\[
T_tf= \Psi( i_{\mathsf{b}}(t)) \quad \mbox{ for all } t \in G \,.
\]
\end{itemize}
Moreover, in this case $F(t)=\alpha(t,x)$ define a group homomorphism with dense range $F : G \to H_f$, and $\Psi$ is an onto group homomorphism.
 \qed
\end{theorem}

Note here that the mapping $\delta_0 : \Cu(G) \to \C$ with
\[
\delta_0(f)=f(0)
\]
is uniformly continuous, and hence so is $\delta_0 \circ \Psi \in C(G_{\mathsf{b}})$. Then, (iv) in Theorem~\ref{CU-equiv} implies that there exists some $g:=\delta_0 \circ \Psi \in C(G_{\mathsf{b}})$ such that
\[
f(-t)= g(i_{\mathsf{b}}(t)) \,.
\]
It follows easily from here that (iv) in Theorem~\ref{CU-equiv} can be replaced by
\begin{itemize}
  \item[(iv')] There exists some $g \in C(G_{\mathsf{b}})$ such that $f=g \circ i_{\mathsf{b}}$.
\end{itemize}

\subsection{Group valued almost periodic functions}
In this section we review the concept of group valued almost
periodic functions, as discussed recently in \cite{LLRSS}.

Let $G,H$ be two LCAG, with $H$ complete. Let $\cH$ be the set of all functions $f: G \to H$.
For each neighbourhood $W \subseteq H$ of $0$ define
\[
U_W := \{ (f,g) \in \cH \times \cH \,:\, f(x)-g(x) \in W \mbox{ for all
} x \in G\} \,,
\]
and let
\[
\cU= \{ U \subseteq \cH \times \cH \,:\, \text{ there exists } W \mbox{ such that } U_W \subseteq U \} \,.
\]
As usual, let $C(G:H)$, $C_{b}(G:H)$ and $\Cu(G:H)$ denote the
subspaces of $\cH$ consisting of continuous, continuous bounded, and
uniformly continuous and bounded functions, respectively. From
\cite[Lem.~6]{LLRSS} we find the following.
\begin{lemma}
\begin{itemize}
\item [(a)] $\cU$ is a uniformity on $\cH$.
\item [(b)] $\cH$ is complete with respect to $\cU$.
\item [(c)] $C(G:H), C_{b}(G:H)$ and $\Cu(G:H)$ are closed in $\cH$.
\end{itemize}
\end{lemma}
Note here that $G$ acts naturally on $\cH$ via the translation
\[
(T_tg)(x)=g(x-t) \,,
\]
and that $\cU$ is $G$-invariant.  While the translation  action of
$G$ is not continuous on $\cH$, it is equicontinuous on $\Cu(G:H)$.
Therefore, Theorem~\ref{thm:apequiv} gives the
following statement (compare \cite[Prop.~9]{LLRSS}).

\begin{theorem}\label{thm-1}
Let $f \in \Cu(G:H)$. Then, the following are assertions equivalent.
\begin{itemize}
  \item[(i)] $f$ is Bohr-type almost periodic.
  \item[(ii)] $f$ is Bochner-type almost periodic.
  \item[(iii)] $H_f:=\overline{ \{ T_tf \,:\, t \in G \}}$ is a compact Abelian group with the addition operation induced by
\[
(T_tf) \oplus (T_sf)= T_{s+t} f\,.
\]
\item[(iv)] There exists a continuous function $\Psi : G_{\mathsf{b}} \to H_f$ such that
\[
T_tf= \Psi( i_{\mathsf{b}}(t)) \quad \mbox{ for all } t \in G \,.
\]
\end{itemize}
Moreover, in this case $F(t)=\alpha(t,x)$ define a group homomorphism with dense range $F : G \to H_f$, and $\Psi$ is an onto group homomorphism. \qed
\end{theorem}

Exactly as before, we can define a continuous mapping $F : C(G:H) \to H$ via $F(h)=h(0)$. Then, for each $f$ satisfying the conditions of Theorem~\ref{thm-1}, the function
$f_{\mathsf{b}}(x):= (F \circ \Psi)(-x)$ belongs to $C(G_{\mathsf{b}}:H)$ and satisfies
\[
f_{\mathsf{b}}(i_{\mathsf{b}}(t)) =f(t) \quad \mbox{ for all } t \in G
\,.
\]

\subsection{Almost periodic functions with values in a Banach space}
Bohr almost periodic functions with values in a Banach space have
attracted particular attention (see e.g. \cite{che,SV}). Note also that this class includes functions of the form
\[
G\longrightarrow \mathcal{H} \,, \qquad t\mapsto T_t f \,,
\]
whenever $\mathcal{H}$ is a Hilbert space, $f$ belongs to
$\mathcal{H}$  and $T$ is a unitary representation of $G$ on
$\mathcal{H}$. This class of functions is of prime importance in the
study of diffraction (see e.g. \cite{Gouere-1,LS}).

Now, the   class of Bohr-almost periodic functions on Banach spaces
is actually a particular case of group valued Bohr-almost periodic
functions discussed in the preceding section.  Indeed, if $(B , \|
\cdot \|)$ is Banach space over $\R$ or $\C$, Bohr almost
periodicity in
\[
BC^0(\R, B):= \{ f : \R \to B \,:\, f \mbox{ bounded and continuous } \}
\]
is simply Bohr almost periodicity in $\Cu(G:H)$ where $G=(\R,+)$ and
$H=(B,+)$. Here we can work in $\Cu(G:H)$ because every Bohr almost
periodic function in $BC(\R,B)$ is uniformly continuous \cite{che}.
So, Theorem \ref{thm-1} applies.

\subsection{Stepanov almost periodicity}
Here we follow closely \cite{Spi}.  Let $G$ be a LCAG and $K
\subseteq G$ be a fixed compact set with non-empty interior, and let
$1 \leq p < \infty$.

Let $BS_K^p(G)$ denote the space of all $f \in L^p_{loc}(G)$ such that
\[
\| f \|_{S_K^p}:= \sup_{y \in G} \left( \frac{1}{|K|} \int_{y+K} |f(t)|^p \dd t \right)^{\frac{1}{p}} < \infty \,.
\]
Then, $(BS_K^p(G), \| \cdot \|_{S_K^p})$ is a Banach space \cite[Prop.~2.2]{Spi}. Now, let $\cU_{step}$ be the uniformity defined by this norm on $BS_K^p(G)$.
Since the Stepanov norm is $G$-invariant by definition, so is $\cU_{step}$. Moreover, the translation action is continuous \cite[Lem.~2.5]{Spi} and hence uniformly continuous on every $H_{f}^{\cU_{step}}$ by Lemma~\ref{lem:7}.

Now, Bohr-type almost periodicity with respect to this uniformity is
called \textbf{Stepanov almost periodicity}. Therefore, Stepanov
almost periodicity is equivalent to Bohr-type almost periodicity
with respect to this uniformity, giving \cite[Prop.~2.7]{Spi}.

\subsection{Weak almost periodicity}
Consider $X=\Cu(G)$ with the weak topology of the Banach space
$(\Cu(G), \| \cdot \|_\infty)$. Note here that the weak topology is
not complete for infinite dimensional vector spaces, and hence is
not complete on $\Cu(G)$ unless $G$ is a finite group. Moreover, the
translation action $T_tf(x)=f(x-t)$ is not equicontinuous.

Because of this, in this case we only get the implications
\begin{align*}
\text{Bochner-type almost periodicity}
    &\implies \text{pseudo Bochner-type almost periodicity} \\
    &\implies \text{Bohr-type almost periodicity} \,.
\end{align*}
This is the
reason why weak almost periodicity for functions is defined via
Bochner-type almost periodicity.

\subsection{Autocorrelation topology}
Let $G$ be a second countable LCAG, and let $\cA=(A_n)$ be a van Hove
sequence (which exists due to the second countability of $G$, see \cite[Prop. B.6]{SS}). Fix an open set $U=-U \subseteq G$
and define
\[
\cD_{U}(G) :=\{ \Lambda \subseteq G \,:\, \Lambda \mbox{ is $U$-uniformly discrete} \} \,.
\]
Next, let $\cA$ be an exhaustive nested van Hove sequence (see \cite{MS} for definitions and properties). Then, $\cA$ induces \cite[Prop.~2.1]{MS} a semi-metric $d_{\cA}$ on $\cD_{U}$ via
\[
d_{\cA}(\Lambda, \Gamma):= \limsup_{n\to\infty} \frac{\card( (\Lambda \Delta \Gamma) \cap A_n)}{|A_n|}  \,,
\]
which becomes a metric on the set $\cD:=\cD_{U}/\equiv$ of equivalence classes under the equivalence relation
\[
\Lambda \equiv \Gamma \Leftrightarrow d_{\cA}(\Lambda, \Gamma)=0 \,.
\]
By \cite[Prop.~2.1 and Cor.~3.10]{MS} $\cD$ is complete and $d$ is $G$-invariant.

Let $\cU$ denote the uniformity defined by $d$. Note here that the uniform discreteness of elements in $\cD_{U}$ imply that the translation action is never continuous. Since  $(\cD, \cU)$ is a complete uniform space, $\cU$ is $G$-invariant and both the topology of $G$ and $\cU$ are metrisable, we can mix the uniformity as in Section~\ref{mix}, and then \cite[Lem.~4.5 and Prop.~4.6]{MS} are simply consequences of Corollary~\ref{cor1}.

\subsection{Vague topology}
Consider now $\XX =\cM(G)$ the space of measures on $G$ equipped
with the vague topology, and the translation action of $G$. Since
the vague topology is complete, Bochner-type and pseudo-Bochner type almost
periodicity are equivalent. Moreover, a measure $\mu \in \XX$ is
Bochner-type almost periodic if and only if it is translation
bounded \cite{BL,SS}.

Now, the translation action of $G$ on $\XX$ is not equicontinuous,
so Bochner and Bohr type almost periodicity may not be equivalent.
Since every translation bounded measure is Bochner-type almost
periodic, it is automatically Bohr-type almost periodic.

Next, consider the measure
\[
\mu:= \sum_{n \in \N} n\,\delta_{5^{n+1}\Z+2 \cdot 5^n}
\]
Now, for each $N$ define
\[
  \mu_N :=\sum_{n =1}^{N-1} n\,\delta_{5^{n+1}\Z+2 \cdot 5^n}  \qquad \text{ and } \qquad
  \nu_N :=\sum_{n =N}^{\infty} n\,\delta_{5^{n+1}\Z+2 \cdot 5^n} \,.
\]
Then, by construction $\mu_N$ is $5^N$ periodic and $\supp(\nu_N) \subseteq 2 \cdot 5^{N}+ 5^{N+1}\Z$. In particular, for all $m \in 5^N +5^{N+1}\Z$, we have
\begin{equation}\label{eq5}
  T_m \mu_N = \mu_N  \,, \quad
  \supp(\nu_N) \cap (-5^N, 5^N)  = \varnothing \quad \text{ and } \quad
  \supp(T_m\nu_N) \cap (-5^N, 5^N)  = \varnothing  \,.
\end{equation}
This immediately implies that $\mu$ is Bohr type almost periodic for the vague topology. Indeed, let $U$ be an entourage for the vague topology. Then, there exists some $\varphi_1, \ldots, \varphi_n \in \Cc(G)$ and $\epsilon >0$ such that
\[
U[\varphi_1, \ldots, \varphi_n;\epsilon]:=\{ (\omega, \varpi) \,:\, | \omega(\varphi_j)- \varpi(\varphi_j) | < \epsilon \text{ for all } 1 \leq j \leq n \} \subseteq U \,.
\]
Let $N$ be so that, for all $1 \leq j \leq N$ we have
\[
\supp(\varphi_j) \subseteq (-5^N, 5^N) \,.
\]
Let $m \in 5^N +5^{N+1}\Z$ be arbitrary. By \eqref{eq5}, we have
\[
T_m \mu_N(\varphi_j)= \mu_N(\varphi_j) \qquad \text{ for all } 1 \leq j \leq n \,,
\]
as well as
\[
T_m \nu_N(\varphi_j)= \nu_N(\varphi_j) \qquad \text{ for all } 1 \leq j \leq n \,.
\]
As $\mu=\mu_N + \nu_N$, we get
\[
(\mu, T_m \mu) \in U \quad \text{ for all } m 5^N +5^{N+1}\Z \,.
\]
This shows that $\mu$ is Bohr-type almost periodic, but, since it is not translation bounded, it is not Bochner-type almost periodic.

\subsection{Product topology}
Let $G$ act via translation of functions on a uniform space of $(X,
\cU)$ of functions on $G$, with the property that $\Cu(G) \subseteq
X$. We can then \textbf{push} the uniformity $\cU$ to
$\cM^\infty(G)$ the following way.

For each $U \in \cU, n \in \N$ and all $\varphi_1,\ldots,\varphi_n \in \Cc(G)$  define
\[
V[U: \varphi_1,\ldots,\varphi_n]:=\{ (\mu, \nu) \in \cM^\infty(G)
\times \cM^\infty(G) \,:\, (\mu*\varphi_j, \nu*\varphi) \in U \mbox{ for
all } 1\leq j \leq n \} \,.
\]
Next, define
\begin{align*}
\cU_{\text{pu}}
    &:= \{ V \subseteq \cM^\infty(G) \times \cM^\infty(G) \,:\,
    \text{there exist } U \in \cU,\, n \in \N,\, \varphi_1,\ldots,\varphi_n
       \in \Cc(G) \\
    &\phantom{XXXXXXXXXXXXXXX} \mbox{ such that } V[U: \varphi_1,\ldots,\varphi_n]
      \subseteq V \} \,.
\end{align*}
Then, $\cU_{\text{pu}}$ is a uniformity on $\cM^\infty(G)$. If $\cU$ is $G$-invariant,so is $\cU_{\text{pu}}$. Moreover, if the translation action is equi-continuous on $\cU$, it is also equi-continuous on $\cU_{\text{pu}}$. Completeness is in general more subtle.

This process allows us carry many results about almost periodicity from functions to measures. We look next at one such example.

\medskip

Now, consider the case when $X=\Cu(G)$ with the uniformity given by the norm $\| \cdot \|_\infty$. The uniformity induced by the push of this uniformity to $\cM^\infty(G)$ is called the \textbf{product uniformity}, and is denoted by $\cU_{p}$. The corresponding topology is the \textbf{product topology for measures} (see \cite{ARMA} for properties). It is easy to see that $\cU_{p}$ is $G$-invariant, and the translation action is equicontinuous.

Now, it is not known if this space, or if $\cM^\infty(G)$ is complete, but this space is quasi-complete (meaning any equi-translation bounded closed subset is complete \cite{ARMA}). This turns out to be enough in this situation. Indeed, if $\mu \in \cM^\infty(G)$, then all measures in $H^{\cU_{p}}_\mu$ are equi-translation bounded, and hence this orbit is complete.
Therefore, we get the standard characterization of $\SAP(G)$, the first three equivalences appearing in \cite{ARMA,MoSt}, while the equivalence to condition (iv) appearing in \cite{LR,LS2,LLRSS}.

\begin{theorem}
Let $\mu \in \cM^\infty(G)$. Then, the following assertions are equivalent.
\begin{itemize}
  \item[(i)] For all $\varphi_1,\ldots,\varphi_n \in \Cc(G)$ and all $\eps >0$ the set
\[
  P_\eps =\{ t \in G \,:\, \| T_t (\varphi_j*\mu)-(\varphi_j*\mu) \|_\infty < \eps \mbox{ for all } 1 \leq j \leq n\}
\]
  is relatively dense.
  \item[(ii)] For all $\varphi_1,\ldots,\varphi_n \in \Cc(G)$ and all $\eps >0$ the set
\[
  P_\eps =\{ t \in G \,:\, \| T_t (\varphi_j*\mu)-(\varphi_j*\mu) \|_\infty < \eps \mbox{ for all } 1 \leq j \leq n\}
\]
  is finitely relatively dense.
  \item[(iii)]$H^{\cU_{p}}_\mu$ is compact in the product topology.
  \item[(iv)] $H^{\cU_{p}}_\mu$ is a compact Abelian group, with the group operation induced by
\[
  (T_t \mu) \oplus (T_s \mu):= T_{t+s} \mu \,.
\]
\item[(v)] There exists a continuous function $\Psi : G_{\mathsf{b}} \to H^{\cU_{p}}_\mu$ such that
\[
T_tf= \Psi( i_{\mathsf{b}}(t)) \quad \mbox{ for all } t \in G \,.
\]
Moreover, in this case $F(t)=\alpha(t,x)$ define a group homomorphism with dense range $F : G \to H^{\cU_{p}}_\mu$, and $\Psi$ is an onto group homomorphism.
\end{itemize}
\end{theorem}

\subsection{Norm and mixed norm topology}
Consider $X =\cM^\infty(G)$, the space of translation bounded measures. Let $K \subseteq G$ be a fixed compact set with non-empty interior. Then, $K$ defines a norm on $X$ \cite{bm} via
\[
\| \mu \|_{K} := \sup_{t \in G} \left| \mu \right|(t+K) \,.
\]
The space $(X, \| \cdot \|_{K})$ is a Banach space \cite{CRS3}. Let
$\cU_{\text{norm}}$ be the uniformity defined on $\cM^\infty(G)$ by this
norm.

Let us first recall the following definition \cite{bm}.

\begin{definition}
A measure $\mu$ in $\cM^\infty(G)$ is called \textbf{norm almost periodic} if $\mu$ is Bohr-type almost periodic in this topology.
\end{definition}

\smallskip

Norm almost periodic pure point measures are appear naturally in the
CPS  \cite{bm,CRS,LR,NS11} and have been
fully characterized in \cite{NS11}.

Note that the translation action is not continuous on $X$, and while
it is continuous on some orbits (for example it is continuous on the
orbit of measures of the form $\mu =f \theta_G$ for $f \in \Cu(G)$)
it is never continuous on the orbit of non-trivial pure point
measures (see Lemma~\ref{lem1} below). In particular, in this case
we have
\begin{align*}
\text{Bochner-type almost periodicity}
    &\iff  \text{pseudo Bochner-type almost periodicity} \\
    &\implies \text{Bohr-type almost periodicity} \,.
\end{align*}

It turns out that for pure point measures, Bochner-type almost
periodicity is not a relevant concept, as we have the following
lemma.

\begin{lemma}\label{lem1}
Let $\mu \in \cM^\infty(G)$ be a pure point measure.
\begin{itemize}
\item[(a)] If the translation action is continuous at a point $s \in G$ on the set $\{  T_t \mu \,:\, t \in G \}$, then $\mu=0$.
\item[(b)] If $G$ is $\sigma$-compact and uncountable and $\mu$ is pseudo Bochner-type almost periodic, then $\mu=0$.
\end{itemize}
\end{lemma}
\begin{proof}
(a) Assume by contradiction that $\mu \neq 0$. Then, since $\mu$ is pure point, there exists some $x \in G$ such that
\[
|\mu(\{x\})| = a >0 \,.
\]
Next, fix some open precompact set $U$ with $s \in U$. Since $\mu$ is a measure, we have
\[
| \mu| (x-s+U) < \infty \,,
\]
and hence, the set
\[
F:= \{ y \in x-s+U \,:\, | \mu(\{y\}) | \geq \frac{a}{2} \}
\]
is a finite set. Let $F'= F \backslash \{ x\}$. Then, the set
\[
V:= U \backslash (s-x+F')
\]
is an open neighbourhood of $s$, and for all $t \in V$ we have
\begin{align*}
| T_s\mu (\{ -s+x \}) -T_t\mu\{ -s+x\} |
    &\geq | T_s\mu (\{ -s+x \})| -| T_t\mu\{ -s+x\} |\\
    & =|\mu (\{ x \})| -| \mu\{ t-s+x\} | \geq a-\frac{a}{2} = \frac{a}{2} \,,
\end{align*}
and hence
\[
\| T_s \mu - T_t \mu \|_{K}  \geq \frac{a}{2} \,.
\]
This contradicts the fact that $T$ is continuous at $s$.

\medskip

\noindent (b) Assume by contradiction that $\mu \neq 0$. The argument is similar to the one in (a).
Since $\mu$ is pure point, there exists some $x \in G$ such that
\[
|\mu(\{x\})| = a >0 \,.
\]
Next, fix some open precompact set $U$ with $0 \in U$ and set $\eps = \frac{a}{2}$. We show that
\[
P_\eps := \{ t \in G \,:\, \| T_t \mu - \mu \|_K < \eps \}
\]
is locally finite and hence countable, by the $\sigma$-compactness of $G$. Since $P_\eps$ is finitely relatively dense, this implies that $G$ is countable, a contradiction.

Indeed, let $K \subseteq G$ be any compact set and let
\[
F:= P_\eps \cap K \,.
\]
We want to show that $F$ is finite. First note that for all $t \in F$ we have $ \| T_t \mu - \mu \|_K < \eps$ and hence
\[
\left| \mu(\{t+x\}) -\mu(\{x \}) \right| < \frac{a}{2} \,.
\]
Since $\mu(\{x \})=a$ this gives $|\mu(\{t+x\})| > \frac{a}{2}$.

It follows from here that
\[
\sum_{t \in F} \frac{a}{2} \leq \sum_{t \in F} \left| \mu(\{t+x\}) \right| \leq \left| \mu \right| (x+K) < \infty \,,
\]
giving that $F$ is finite, and completing the proof.
\end{proof}

Since $\cU_{\text{norm}}$ is given by a metric, and $G$-invariant, we can define the mixed uniformity from Section~\ref{mix}. We will denote the mixed uniformity by $\cU_{m-n}$
 and refer to the topology induced by this uniformity as the \textbf{mixed-norm topology} for measures.

Now, Corollary~\ref{cor1} has the following consequence.

\begin{coro}
Let $G$ be metrisable, and let $\mu \in \cM^\infty(G)$. Then, the following assertions are equivalent.
\begin{itemize}
  \item[(i)] $\mu$ is norm almost periodic.
  \item[(ii)] $H^{\cU_{m-n}}_\mu$ is compact.
  \item[(iv)] $H^{\cU_{m-n}}_\mu$ is a compact  Abelian group, with the group operation induced by
  \[
  (T_t \mu) \oplus (T_s \mu):= T_{t+s} \mu \,.
  \]
 \item[(v)] There exists a continuous function $\Psi : G_{\mathsf{b}} \to H^{\cU_{m-n}}_\mu$ such that
\[
T_tf= \Psi( i_{\mathsf{b}}(t)) \, \mbox{ for all } t \in G \,.
\]
\end{itemize}
Moreover, in this case $F(t)=\alpha(t,x)$ define a group homomorphism with dense range $F : G \to H^{\cU_{m-n}}_\mu$, and $\Psi$ is an onto group homomorphism.
\end{coro}

 \subsection*{Acknowledgments} NS was supported by NSERC with grant 2020-00038. TS was supported by the German Research Foundation (DFG) via the CRC 1283.

\end{document}